\documentclass[11pt]{amsart}
\usepackage{stmaryrd}
\usepackage{tikz-cd}
\usepackage[utf8]{inputenc}
\usepackage{fullpage, mathrsfs, mathtools, amssymb, tikz,bm, amscd, scrextend,  verbatim, enumerate, amsmath, amsthm}
\usepackage{extarrows}
\usepackage{wasysym, mathtools, mathrsfs}
\usepackage[colorlinks=true, linkcolor=red!80!black, urlcolor=purple, citecolor=blue!70!black]{hyperref}
\usepackage{microtype}
\usepackage{verbatim}

\usepackage[mathscr]{euscript}
\usepackage[utf8]{inputenc}
\usepackage{amsmath,amssymb,amsthm, graphics, comment,bm}
\usepackage{mathtools}
\usepackage{amsthm}
\usepackage{diagmac2}
\usepackage{tikz-cd}
\usepackage{textcomp}
\usepackage{color}
\usepackage[alphabetic]{amsrefs}
\usepackage[all,cmtip,color,matrix,arrow]{xy}
\usepackage{yfonts}
\usepackage{tikz}
\usepackage{stmaryrd}
\usepackage{mathrsfs}
\usepackage{enumerate}
\usepackage{hyperref}
\hypersetup{
	colorlinks=true,
	linkcolor=blue,
	filecolor=magenta,      
	urlcolor=cyan,
}
\usepackage[mathscr]{euscript}
\usepackage{cleveref}

\usetikzlibrary{matrix, arrows}

\input xy
\xyoption{all} 

\usepackage{esvect}

	\newtheorem{theorem}{Theorem}[section]
	\newtheorem{lemma}[theorem]{Lemma}
	\newtheorem{prop}[theorem]{Proposition}
	\newtheorem{cor}[theorem]{Corollary}

	\theoremstyle{definition}
	\newtheorem{definition}[theorem]{Definition}

\newtheorem{remark}[theorem]{Remark}

\theoremstyle{remark} 

\newtheorem{notation}{Notation}
\newtheorem{observation}{Observation}

\newcounter{step}

\usepackage{faktor}\usepackage{amsmath}\usepackage{amssymb}

\makeatletter
\DeclareRobustCommand*{\mfaktor}[3][]
{
   { \mathpalette{\mfaktor@impl@}{{#1}{#2}{#3}} }
}
\newcommand*{\mfaktor@impl@}[2]{\mfaktor@impl#1#2}
\newcommand*{\mfaktor@impl}[4]{
   \settoheight{\faktor@zaehlerhoehe}{\ensuremath{#1#2{#3}}}%
   \settoheight{\faktor@nennerhoehe}{\ensuremath{#1#2{#4}}}%
      \raisebox{-0.5\faktor@zaehlerhoehe}{\ensuremath{#1#2{#3}}}%
      \mkern-4mu\diagdown\mkern-5mu%
      \raisebox{0.5\faktor@nennerhoehe}{\ensuremath{#1#2{#4}}}%
}

\newcommand{\bP}{{\mathbb P}}
\newcommand{\bA}{{\mathbb A}}
\newcommand{\bQ}{{\mathbb Q}}

\newcommand{\calO}{\mathcal O}
\newcommand{\calM}{\mathcal M}
\newcommand{\calH}{\mathcal H}

\newcommand{\calI}{\mathcal I}

\newcommand{\calB}{\mathcal B}
\newcommand{\calL}{\mathcal L}

\newcommand{\calW}{\mathcal W}

\newcommand{\calU}{\mathcal U}
\newcommand{\calX}{\mathcal X}
\newcommand{\calC}{\mathcal C}
\newcommand{\calP}{\mathcal P}

\newcommand{\calY}{\mathcal Y}
\newcommand{\calZ}{\mathcal Z}

\newcommand{\calD}{\mathcal D}
\newcommand{\calV}{\mathcal V}

\newcommand{\bmu}{\boldsymbol{\mu}}

\newcommand{\sG}{\mathscr{G}}

\newcommand{\sK}{\mathscr{K}}

\newcommand{\sS}{\mathscr{S}}

\newcommand{\sY}{\mathscr{Y}}

\newcommand{\bG}{\mathbb{G}}

\newcommand{\Spec}{\mathrm{Spec}}

\newcommand{\Pic}{\mathrm{Pic}}

\newcommand{\oH}{\operatorname{H}}
\newcommand{\Hom}{\operatorname{Hom}}

\author{Giovanni Inchiostro}

\title{Moduli of genus one curves with two marked points as a weighted blow-up}
\begin{document}

\begin{abstract}
We give an explicit description of $\overline{\calM}_{1,2}$ as a weighted blow-up of a weighted projective stack. We use this description to compute the Brauer group of $\overline{\calM}_{1,2;S}$ over any base scheme $S$ where 6 is invertible, and the integral Chow rings of 
$\overline{\calM}_{1,2}$ and $\calM_{1,2}$.
\end{abstract}
\maketitle
\section{Introduction} The moduli spaces of curves have been an object of central interest in algebraic geometry since the
1800 with Riemann and Poincaré. Understanding its global geometry and the numerical invariants
generated a vaste amount of literature.

The main focus of this paper is the moduli space of 2-pointed genus 1 curves, and their Deligne-Mumford compactification $\overline{\calM}_{1,2}$. Our first goal is to revisit a description by Massarenti in \cite{Mas14} of the coarse moduli space of $\overline{\calM}_{1,2}$, as a weighted blow-up of a weighted projective space, to get an analogous description for the moduli stack over $\mathbb{Z}[\frac{1}{6}]$:
\begin{theorem}\label{theorem:intro:weighted}Let $S$ be a scheme where $6$ is invertible, and let $\calP(2,3,4)_S$ be the weighted projective stack over $S$, with weights 2, 3, and 4. The map $\calP(2,3,4)_S\to S$ has a section $z:S\to \calP(2,3,4)_S$ such that the weighted blow-up of $\calP(2,3,4)_S$ at $z$ with weights 4 and 6 is isomorphic to $\overline{\calM}_{1,2;S}$.
\end{theorem}
In the special case where $S=\Spec(k)$ is a field, we obtain an isomorphism between $\overline{\calM}_{1,2}$ and a weighted blow-up of a point in a weighted projective stack.

One can understand Theorem \ref{theorem:intro:weighted} as follows. First, we identify $\calP(2,3,4)$ with a moduli space of genus 1 curves with two marked points (see Subsection \ref{subsection_desc_of_the_space}). This in turn produces a rational map $\calP(2,3,4)\dashrightarrow \overline{\calM}_{1,1}$, defined away from a closed (schematic) subset $Z\subseteq \calP(2,3,4)$, with fibers being elliptic curves without a point (see Lemma \ref{lemma_fibers_elliptic_curves_without_a_pt}). We prove  that a weighted blow-up with weights 4 and 6 resolves the indeterminacy. This will give a map from such a blow-up $B_Z^{(4,6)}\calP(2,3,4)\to \overline{\calM}_{1,1}$. In Theorem \ref{theorem_iso_weighted_blowup_M12} we identify this blow-up with the morphism from the universal family $\overline{\calM}_{1,2} \to \overline{\calM}_{1,1}$.

The main advantage of this presentation is its explicit nature, and exploiting its relation with a (relatively) simpler space (namely $\calP(2,3,4)$), we obtain the following:

\begin{theorem}Let $S$ be a scheme where 6 is invertible.
If we denote by $Br'(\cdot)$ the cohomological Brauer group and by $A^*(\cdot)$ the Chow ring, we have:
\begin{enumerate}
    \item $Br'(\overline{\calM}_{1,2;S}) = Br'(S)$,
    \item if $S$ is a field, $A^*(\overline{\calM}_{1,2}) = \mathbb{Z}[x,y]/(24x^2 + 24 y^2, xy)$, 
    \item if $S$ is a field, $A^*(\calM_{1,2}) = \mathbb{Z}[t]/(12t)$, and
    \item we can identify $A^*(\overline{\calM}_{1,2}) \to A^*(\calM_{1,2})$ with the map $y\mapsto 0$ and $x\mapsto t$.
\end{enumerate}
\end{theorem}
We observe that the description of $A^*(\overline{\calM}_{1,2}) = \mathbb{Z}[x,y]/(24x^2 + 24 y^2, xy)$ was obtained, by other means, in \cite[Theorem 2.6]{di2021stable}, by relating $\overline{\calM}_{1,2}$ to a stack of so-called $A_2$ curves.

\subsection{Background on Brauer groups and integral Chow rings of moduli problems}
In this subsection we bring to the reader's attention some of the existing literature on Brauer groups and Chow rings of moduli problems.

While there is an extensive amount of literature on Brauer groups of schemes, Brauer groups of moduli problems are much harder to compute, especially when the base scheme is not a field, and only few examples are known. For example, in \cite{period_index_max} Lieblich computes the Brauer group of $B\bmu_k$ over fields where $k$ is invertible. In \cite{Brauer_group_m11}, Antieau and Meier compute the Brauer group of $\calM_{1,1;S}$ over a wide range of base schemes $S$, and Shin in \cite{Minseon} computes the Brauer group of weighted projective stacks over an arbitrary base scheme $S$. Similarly, Fringuelli and Pirisi \cite{FringuelliPirisi} compute the Brauer groups of $\calM_{g,n}$ and of certain moduli of vector bundles over a field, and using cohomological invariants Di Lorenzo and Pirisi compute the Brauer group of the moduli of hyperelliptic curves in \cite{DLP}.

Similarly to Brauer groups, Chow rings of moduli problems, especially with integral coefficients, are often quite difficult to compute. For example, only a few examples have been computed in full, for moduli stacks of curves $\calM_{g,n}$ and $\overline{\calM}_{g,n}$ with $g$ and $n$ small (see \cites{Vis, Lar, DLV, di2021stable}) or for other moduli problems
(see \cites{ EdFu, DL, DLFV, CDLI}).

\subsection{Conventions.} We will work over a scheme $S$ where 6 is invertible. "Points" will always be geometric points, and "fibers" will always be fibers over geometric points. Given an equivalence relation on a set $\sS$ and given $a\in \sS$, we will denote by $[a]$ the equivalence class of $a$.  Given a tuple of $n+1$ integers $a_0,...a_n$ and a scheme $S$ we denote by $\mathcal{P}_S (a_0,...,a_n)$ the weighted projective stack with weights $a_1,...,a_n$ over $S$. This is defined as the quotient $[\bA^{n+1}\smallsetminus \{0\}/\bG_m]$ where the action of $\bG_m$ has weight $a_i$ on $x_i$. We will assume that all the flat morphisms have equidimensional fibers. 

\subsection{Acknowledgements} I thank Dan Abramovich, Jarod Alper, Dori Bejleri, Andrea Di Lorenzo and Siddharth Mathur for helpful discussions. I thank Minseon Shin for suggesting Grothendieck existence theorem and Artin approximation for proving Theorem \ref{teo_brauer}, and for the helpful discussions that followed.  I also thank the referee
for carefully reading the draft and giving insightful feedback. 

\section{$\overline{\calM}_{1,2}$ as a weighted blowup}
\subsection{Rational map $\calP(2,3,4)\dashrightarrow \calP(4,6)$}\label{subsection_desc_of_the_space} In this subsection we define a map $\calP(2,3,4)\dashrightarrow \calP(4,6)$. Resolving the indeterminacy of this map leads to our description of $\overline{\calM}_{1,2}$. We work over $R:= \mathbb{Z}[\frac{1}{6}]$.

Consider the action of $\bG_m$ on  $\mathbb{A}^2$ and $\bA^3$ with weights $(4,6)$ and $(2,3,4)$ respectively. Let $\alpha_2,\alpha_3,\alpha_4$ be the coordinates on $\bA^3$ and $\beta_4,\beta_6$ those on $\bA^2$. Then there is an equivariant morphism $F:\mathbb{A}^3\to \mathbb{A}^2$ defined by sending $(\alpha_2, \alpha_3, \alpha_4)\mapsto (\alpha_4, \alpha_3^2 - \alpha_2^3 -\alpha_2\alpha_4)$.
This equivariant map induces a morphism $f: [\mathbb{A}^3/\bG_m] \to [\mathbb{A}^2/\bG_m]$.
\begin{lemma}\label{lemma_fibers_elliptic_curves_without_a_pt}
The morphism $f$ is representable, its geometric fibers are genus 1 curves without a point, and the fiber of $(\beta_4,\beta_6)$ is isomorphic to $V(y^2-x^3 - \beta_4x -\beta_6)\subseteq \mathbb{A}^2_{x,y}$.
\end{lemma}
\begin{proof}
From \cite[Exercise 10.F]{Ols16}, the following diagram is cartesian:
$$\xymatrix{\mathbb{A}^3 \ar[d] \ar[r]^F & \mathbb{A}^2\ar[d]\\[\mathbb{A}^3/\bG_m] \ar[r] & [\mathbb{A}^2/\bG_m].}$$
In particular, the geometric fibers of $f$ are those of $F$, as $\mathbb{A}^3\to [\mathbb{A}^3/\bG_m]$ is surjective. Therefore it suffices to check the desired statement for $F$. The fiber of a point $(\beta_4,\beta_6)$ is $V(\alpha_4 - \beta_4, \alpha_3^2 - \alpha_2^3 -\alpha_4\alpha_2 - \beta_6) \subseteq \bA^3_{\alpha_2, \alpha_3, \alpha_4}$ which in turn, by setting $ x = \alpha_2$ and $y = \alpha_3$, is isomorphic to $V(y^2-x^3 - \beta_4x -\beta_6)\subseteq \mathbb{A}^2_{x,y}$.
\end{proof}
We now explain the relevance of $f$ in terms of explicit equations for marked elliptic curves. Following \cite[Section III.3]{Silverman}, one can check that if $k$ is a field with 6 invertible, any elliptic curve with two marked points can be written as 
\begin{align}y^2z + a_3yz^2 = x^3 + a_2x^2z + a_4xz^2\subseteq \bP^2\times \bA^3\end{align}
where $a_2$, $a_3$ and $a_4$ are coordinates on $\mathbb{A}^3$, and the two marked points are $[0,1,0]$ and $[0,0,1]$. 
One can check that, given six choices $(a_2,a_3,a_4)$ and $(a_2',a_3',a_4')$, the two elliptic curves \begin{center}$y^2z + a_3yz^2 = x^3 + a_2x^2z + a_4xz^2$ and $y^2z + a_3'yz^2 = x^3 + a_2'x^2z + a_4'xz^2$\end{center} are isomorphic if and only if there is an element $\lambda \in \bG_m$ such that $a_2 = \lambda^2 a_2'$, $a_3 = \lambda^3 a_3'$ and $a_4 = \lambda^4 a_4'$.
Therefore, it is natural to consider an open substack of $\calP(2,3,4)$ as a moduli space for ellpitic curves with two marked points.

Similarly, recall that every elliptic curve can be written as $y^2z = x^3 + \beta_4xz^2 + \beta_6z^3$. This in turn gives the isomorphism $\overline{\calM}_{1,1}\cong \calP(4,6)$ (see for example \cite[Remark at page 627]{EG98}). Therefore, since $\calM_{1,2}$ maps to $\overline{\calM}_{1,1}$, it is natural to
try to obtain such a map by writing $y^2z + a_3yz^2 = x^3 + a_2x^2z + a_4xz^2$ in its Weierstrass form.

 If we perform the change of variables $X:= x + \frac{a_2}{3}$, $Y = y + \frac{a_3}{2}$ and $Z=z$, equation (1) becomes:
$$Y^2Z = X^3 + Z^2X\left(a_4 - \frac{a_2^2}{3}\right) + Z^3\left(\frac{a_3^2}{4} - \frac{a_2^3}{27} + \frac{a_2^3}{9} - \frac{a_4a_2}{3}\right).$$
Defining $\alpha_2 := \frac{a_2}{3}$, $\alpha_3 := \frac{a_3}{2}$ and $\alpha_4:=a_4 - \frac{a_2^2}{3}$; we obtain the desired Weierstrass equation:
$$Y^2Z = X^3 + \alpha_4XZ^2 + Z^3(\alpha_3^2 - \alpha_2^3 -\alpha_2\alpha_4).$$
In particular, $\beta_4 = \alpha_4$ and $\beta_6 = \alpha_3^2 - \alpha_2^3 -\alpha_2\alpha_4$: this explains the definition of $f$.

\subsection{Weighted blow-up of a smooth point on a surface}\label{subsection:weighted:blowup} 

To resolve the indeterminacy of the morphism $\calP(2,3,4)\dashrightarrow \calP(4,6)$, we need to perform a weighted blow-up. We will perform this transformation on a Zariski open subset of $\calP(2,3,4)$ isomorphic to $\Spec(R[x,y]_{x-1})$, so it is useful to understand the weighted blow-up of the ideal $(x,y)$ in $\Spec(R[x,y])$.

Consider the action of $\bG_m$ on $\Spec(R[x,y,u])$ with weights $(w_1,w_2,-1)$ with $w_1, w_2>0$. The resulting quotient stack has a good moduli space map $[\mathbb{A}^3/\bG_m]\to \Spec((\mathbb{A}^3)^{\bG_m})$
given by taking the invariant ring $R[x,y,u]^{\bG_m}$ (see \cite[Definition 4.1]{GMS_jarod} for the definition of good moduli space).
\begin{observation}
The ring of invariants $R[x,y,u]^{\bG_m}$ is $R[u^{w_1}x,u^{w_2}y]$. Indeed, consider an invariant polynomial $Q = \sum_{i,j,k}a_{i,j,k}x^iy^ju^k\in R[x,y,u]$, let $i_0,j_0,k_0$ such that $a_{i_0,j_0,k_0}\neq 0$, and let $p\gg 0$ be a prime number. We will denote by $[a_{i_0,j_0,k_0}]$ the reduction of $a_{i_0,j_0,k_0}$ modulo $p$. Then as $Q$ is invariant, for every $\lambda \in \mathbb{F}_p^*$, one has $[a_{i_0,j_0,k_0}] = \lambda^{i_0w_1 + j_0w_2 - k_0} [a_{i_0,j_0,k_0}]$. As we can choose $p$ big enough, one has $k_0=i_0w_1 + j_0w_2$, or in other terms $a_{i_0,j_0,k_0}x^{i_0}y^{j_0}u^{k_0} = a_{i_0,j_0,k_0}(xu^{w_1})^{i_0}(yu^{w_2})^{j_0}$.
\end{observation}

\begin{remark}\label{remark quote pt 2}
Identifying the ring of invariants $R[x,y,u]^{\bG_m}$ with $\bA^2$, we can check that:
\begin{enumerate}
    \item The good moduli space map $\pi:[\mathbb{A}^3/\bG_m]\to \mathbb{A}^2$ is an isomorphism on $\pi^{-1}(\mathbb{A}^2\smallsetminus\{0\})$,
    \item The map $f:[\mathbb{A}^3/\bG_m]\to [\mathbb{A}^2/\bG_m]$ defined by $[X,Y,U]\mapsto [X,Y]$, where the action on $\mathbb{A}^2$ is with weights $w_1,w_2$, extends the map $\mathbb{A}^2 \to [\mathbb{A}^2/\bG_m]$ which sends $(a,b)\mapsto [a,b]$, and
    \item The morphism $f$ is smooth and representable.
\end{enumerate}
\end{remark} 
To check point (3) it suffices to observe that being smooth and representable are properties that can be checked smooth locally, and that the following diagram is cartesian, with surjective smooth vertical arrows:
$$\xymatrix{\mathbb{A}^3 \ar[d] \ar[r]^{(a,b,c)\mapsto (a,b)} & \mathbb{A}^2 \ar[d] \\[\mathbb{A}^3/\bG_m] \ar[r]_f & [\mathbb{A}^2/\bG_m].}$$

If we denote by $\sS$ the points of the form $[0,0,u]\in [\mathbb{A}^3/\bG_m]$, we have
$\phi: [\mathbb{A}^3/\bG_m]\smallsetminus \sS \to \calP(w_1,w_2)$.
The following will be the definition that we will adopt for a weighted blow-up of the origin in $\bA^2_S$, and one can give a similar definition for a weighted blow-up of an open neighbourhood of the origin. In Remark \ref{remark_dan_definition_weighted} we observe that it agrees with other definitions in the literature in the case, for example, where the base scheme $S$ is an algebraically closed field of characteristic 0.
\begin{definition}
We will denote  $B^{w_1,w_2}_{(0,0)}\mathbb{A}^2 := [\mathbb{A}^3/\bG_m]\smallsetminus \sS$ the weighted blow-up of $\mathbb{A}^2$ at $(0,0)$.
\end{definition}
\begin{remark}\label{remark map A2 to Pw1,w2} The construction above gives a commutative diagram as follows:
$$\xymatrix{\mathbb{A}^2\smallsetminus \{(0,0)\} \ar[r] \ar@{_{(}->}[dr]& B^{w_1,w_2}_{(0,0)}\mathbb{A}^2 \ar[d] \ar[r]^{\phi} & \calP(w_1,w_2) \ar@{_{(}->}[d]\\ & \mathbb{A}^2 \ar[r] & [\bA^2/\bG_m] }$$
where the top map $\mathbb{A}^2\smallsetminus\{(0,0)\} \to \calP(w_1,w_2)$ sends $(a,b)\mapsto [a,b]$.
\end{remark}

\begin{observation}\label{obs_blowup_has_a_section}
The exceptional divisor of $B^{w_1,w_2}_{(0,0)}\mathbb{A}^2 \to \mathbb{A}^2$ is isomorphic to $\calP(w_1,w_2)$, and its inclusion
$\iota: \calP(w_1,w_2)\hookrightarrow B^{w_1,w_2}_{(0,0)}\mathbb{A}^2$ gives a section to $\phi : B^{w_1,w_2}_{(0,0)}\mathbb{A}^2 \to \calP(w_1,w_2)$. Moreover, $\phi$ is a line bundle. Indeed, the map $f:[\bA^3/\bG_m] \to [\bA^2/\bG_m]$ is a line bundle since it comes from a map $\bA^3\cong \bA^2\times V\to \bA^2$ for a one dimensional representation $V$ of $\bG_m$, and $\phi$ is the restriction of $f$ to $\calP(w_1,w_2)$. Specifically, it is $\calO_{\calP(w_1,w_2)}(-1)$ and $\iota$ is the zero section. In particular, as if $\calL\to \calX$ is a line bundle with zero section $\sigma:\calX\to \calL$ having ideal sheaf $\calI$, then $\sigma^*\calI \cong \calL^{-1}$; the restriction of the ideal sheaf of the exceptional divisor to $\calP(w_1,w_2)$ is isomorphic to $\calO_{\calP(w_1,w_2)}(1)$.
\end{observation}
\begin{observation}\label{Oss_pull_back_exceptional_gives_iso_on_pic}
From Observation \ref{obs_blowup_has_a_section}, the inclusion of the exceptional divisor $\calP(w_1,w_2)\to B_{(0,0)}^{w_1,w_2}\bA^2$ induces a pull-back map on Picard groups which is an isomorphism
\end{observation}
\begin{lemma}
With the notations above, the morphism  $\pi:B^{w_1,w_2}_{(0,0)}\mathbb{A}^2 \to \mathbb{A}^2$ is proper.
\end{lemma}
\begin{proof}The product of $\phi$ and $\pi$ gives a map $F:B^{w_1,w_2}_{(0,0)}\mathbb{A}^2 \to \mathbb{A}^2\times \calP(w_1, w_2)$. We show that $\pi$ is proper by showing that this map is finite.

Consider the two maps $\mathbb{A}^3\to \bA^2$, $(x,y,z)\mapsto (x,y)$ and $\mathbb{A}^3\to \bA^2$, $(x,y,z)\mapsto (xz^{w_1}, yz^{w_2})$.
They give a map to the fibred product $\mathbb{A}^3\to \bA^4$, given by $(x,y,z)\mapsto (x,y,z^{w_1}x,z^{w_2}y)$. If we denote the coordinates on $\mathbb{A}^4$ by $(X,Y,z,t)$, take the open subset $\calU\subseteq \mathbb{A}^4$ whose complement is $V(X,Y)$. Its preimage, which we denote by $\calV$, is the open subset whose complement is $V(x,y)$. Therefore there is a morphism $\calV\to \calU$.
By restricting this map to $\calV\smallsetminus D(x)\to \calU\smallsetminus D(X)$ and $\calV\smallsetminus D(y)\to \calU\smallsetminus D(Y)$ can check that $\calV\to \calU$ is a finite map.

Let $\bG_m$ act on $\bA^3$ with weights $(w_1, w_2, -1)$ and on $\bA^4$ with weights $(w_1, w_2, 0, 0)$, the open subsets $\calU$ and $\calV$ are invariant, and the corresponding map $\calV\to \calU$ is equivariant. We can identify $B^{w_1,w_2}_{(0,0)}\bA^2 \cong [\calV/\bG_m]$, and $\calP(w_1,w_2)\times \bA^2 \cong [\calU/\bG_m]$, and the corresponding map $[\calV/\bG_m]\to [\calU/\bG_m]$ with $F$. Since a weighted projective stack is proper, and since proper morphisms are stable under base change, the projection $[\calU/\bG_m]\to \bA^2$ is proper. Since a finite morphism is proper, also $[\calV/\bG_m]\cong B^{w_1,w_2}_{(0,0)}\bA^2 \to \bA^2$ is proper. 
\end{proof}
\begin{remark}\label{remark_dan_definition_weighted}
We can check that our description agrees with the classical one (for example, the one in \cite[Section 3]{ATW19}) over a field of characteristic 0 containing the $w_1$ and $w_2$ roots of unity, by giving two charts isomorphic to the ones in \cite[Section 3.4]{ATW19}. We construct only the one corresponding to $x\neq 0$, the one corresponding to $y \neq 0$ is analogous. Consider the map $\alpha:\mathbb{A}^2 \to \mathbb{A}^3$, $(a,b)\mapsto (1,a,b)$ and the group
homomorphism $\beta: \bmu_{w_1}\to \bG_m$, $\xi \mapsto \xi^{-i}$ where $\xi$ is a $w_1$-th root of 1. If $\bmu_{w_1}$ acts on $\mathbb{A}^2$ with weights $(-w_1, 1)$, the maps $\alpha$ and $\beta$ gives the desired chart $[\mathbb{A}^2/\bmu_{w_1}]\to [\mathbb{A}^3/\bG_m]$ as in \cite[Section 3.4]{ATW19}.
\end{remark}

\subsection{Weighted blow-up for $\overline{\calM}_{1,2}$} We now return to the setting of Subsection \ref{subsection_desc_of_the_space}.

The morphism $f:[\bA^3/\bG_m]\to [\bA^2/\bG_m]$ at the end of Section \ref{subsection_desc_of_the_space} gives a rational map $\calP(2,3,4)\dashrightarrow \calP(4,6).$ This rational map is not defined at the fiber of $f$ at $[0,0]$, namely at the the locus $\{[  s^2, s^3, 0]: \text{ }s\in \bG_m\} \subseteq \calP(2,3,4)$. 
\begin{notation}\label{notation_point_z}
We will denote by $Z$ the locus $\{[  s^2, s^3, 0]: \text{ }s\in \bG_m\} \subseteq \calP(2,3,4)$. Observe that over a field, this is the point $[1,1,0]$.
\end{notation}
We will now resolve the indeterminacy of $\calP(2,3,4)\dashrightarrow \calP(4,6)$. As mentioned at the beginning of Subsection \ref{subsection:weighted:blowup}, it is useful to understand the local description of $\calP(2,3,4)$ around $Z$:

\begin{lemma}\label{lemma_local_desc_of_calP_around_z}
There is a Zariski neighbourhood $U$ of $Z$ and an isomorphism $\Spec(R[x,y]_{x-1})\cong U$, such that the composition $\Spec(R[x,y]_{x-1})\cong U \xrightarrow{f}[\mathbb{A}^2/\bG_m]$ sends a point $(a,b)\mapsto [a,b]$.
\end{lemma}\begin{proof}
The coarse space of $\calP(2,3,4)$ is the usual weighted projective space $P(2,3,4)$, let $\pi$ be the coarse space morphism.
Let $X_2$, $X_3$, $X_4$ be the coordinates of degree 2, 3 and 4 in $P(2,3,4)$. Given a homogeneous polynomial $f$ denote by $D(f)$ the open subset of $P(2,3,4)$ where $f \neq 0$, and with $\calD(f):=\pi^{-1}(D(f))$ its preimage in $\calP(2,3,4)$. Then $Z$ lies in $D(X_2)\cong \mathbb{A}^2_{x,y}$ where $(x,y)=(\frac{X^2_3}{X_2^3},\frac{X_4}{X_2^2})$. The only points with non-trivial stabilizers on $\calD(X_2)$ are the points of the form $[1,0,x]$, and they have a $\bmu_2$-stabilizer. Then $\calD(X_2X_3)$ is an algebraic space, and so $\calD(X_2X_3) \cong D(X_2X_3)\cong \Spec(R[x,y]_x)$.

The rational morphism $\calP(2,3,4)\to [\mathbb{A}^2/\bG_m]$ restricted to $\calD(X_2,X_3)$ and composed with $\calD(X_2,X_3) \cong \Spec(R[x,y]_x)$ gives $\Spec(R[x,y]_x)\to  [\mathbb{A}^2/\bG_m]$, which sends $(a,b)\mapsto [b,a-b-1]$. Then it suffices to perform the change of variables $x \mapsto y$ and $y\mapsto x+y+1$ to prove the result.
\end{proof}
\begin{prop}
If we perform a blow-up $B_Z^{(4,6)}\calP(2,3,4) \to \calP(2,3,4)$ with weights 4 and 6 at $Z$, we can lift the rational map $\calP(2,3,4)\dashrightarrow \calP(4,6)$ to get 
 $\pi:B_Z^{(4,6)}\calP(2,3,4) \to \calP(4,6)$:
 $$\xymatrix{B_Z^{(4,6)}\calP(2,3,4)  \ar[d] \ar@{.>}[rr]^\pi & & \calP(4,6)\ar@{^{(}->}[d]\\ \calP(2,3,4) \ar@{^{(}->}[r] \ar[r] & [\mathbb{A}^3/\bG_m]\ar[r] & [\mathbb{A}^2/\bG_m] }$$
\end{prop}
\begin{proof}As the indeterminacy locus of $\calP(2,3,4)\to \calP(4,6)$ is $Z$, to construct $\pi$ we can work Zariski locally around $Z$. From Lemma \ref{lemma_local_desc_of_calP_around_z} there is a Zariski open neighbourhood $Z\in U\hookrightarrow \calP(2,3,4)$ of $Z$ as follows: \begin{equation}\label{diagram for lifting}
    \xymatrix{\calP(2,3,4)\ar[rr]^-{f|_{\calP(2,3,4)}} && [\bA^2/\bG_m] \\ U \ar@{^{(}->}[u]^{\calD(X_2X_3)} \ar[rr]^-{f|_U} && [\bA^2/\bG_m]\ar[u]_{\operatorname{Id}} \\ \bA^2_{x-1}\ar[u]^{\text{ Lemma } \ref{lemma_local_desc_of_calP_around_z}}_\cong\ar[rr]_-{(a,b)\mapsto [a,b]} && [\bA^2/\bG_m]. \ar[u]_{\operatorname{Id}}} 
\end{equation} Then in a neighbourhood of $Z$, the map $\calP(2,3,4)\to [\bA^2/\bG_m]$ agrees with the restriction to $\bA^2_{x-1}$ of the map $\bA^2\to [\bA^2/\bG_m]$ described in point (2) of Remark \ref{remark quote pt 2}. 
Therefore, proceeding as in Remark \ref{remark map A2 to Pw1,w2}, we can perform a blow-up at $(0,0)$ in $\bA^2_{x-1}$ with weights $(4,6)$ to lift $(a,b)\mapsto [a,b]$:
\begin{equation}
    \xymatrix{&\calP(2,3,4)\ar[rr]^-{f|_{\calP(2,3,4)}} && [\bA^2/\bG_m] \\& U \ar@{^{(}->}[u]^{\calD(X_2X_3)} \ar[rr]^-{f|_U} && [\bA^2/\bG_m]\ar[u]_{\operatorname{Id}} \\ &\bA^2_{x-1}\ar[u]^{\text{ Lemma } \ref{lemma_local_desc_of_calP_around_z}}_\cong\ar[rr]_-{(a,b)\mapsto [a,b]} && [\bA^2/\bG_m]. \ar[u]_{\operatorname{Id}}\\\bA^2_{x-1}\smallsetminus \{(0,0)\}\ar[r]\ar@{^{(}->}[ur]&
    B_{(0,0)}^{4, 6}\bA^2_{x-1} \ar[u]\ar[rr] && \calP(4,6)\ar@{^{(}->}[u]} 
\end{equation}
\end{proof}

\begin{observation}
Observe that $B_Z^{(4,6)}\calP(2,3,4)\to \calP(2,3,4)$ and $\calP(2,3,4)\to \Spec(\mathbb{Z}[\frac{1}{6}])$ are proper, so also $B_Z^{(4,6)}\calP(2,3,4)\to \Spec(\mathbb{Z}[\frac{1}{6}])$ is proper.
\end{observation}

\subsubsection{Isomorphism $B_Z^{(4,6)}\calP(2,3,4) \cong \overline{\calM}_{1,2}$} We are ready to prove that $B_z^{(4,6)}\calP(2,3,4) \cong \overline{\calM}_{1,2}$. First we produce a morphism $B_z^{(4,6)}\calP(2,3,4) \to \overline{\calM}_{1,2}$ by studying $\pi:B_Z^{(4,6)}\calP(2,3,4) \to \calP(4,6)$.
\begin{prop}\label{prop_family_of_ellpitic} The morphism
$\pi$ is a family of genus 1 curves with a section.  
\end{prop}
\begin{proof}
We first prove that $\pi$ is a family of genus 1 curves. Since $B_Z^{(4,6)}\calP(2,3,4)$ is proper, $\pi$ is proper. From Lemma \ref{lemma_fibers_elliptic_curves_without_a_pt} the map
$\pi$ is representable away from the exceptional divisor of $p:B_Z^{(4,6)}\calP(2,3,4) \to \calP(2,3,4)$, and from the results of Subsection \ref{subsection:weighted:blowup} it is representable around the exceptional divisor. Therefore $\pi$ is representable. From Subsection \ref{subsection:weighted:blowup}, the fibers of $\pi$ are smooth along the exceptional divisor. Moreover from Lemma \ref{lemma_fibers_elliptic_curves_without_a_pt}, they are genus 1 curves without a point when we restrict them to $\calP(2,3,4)\smallsetminus Z$. Therefore the fibers are genus 1 curves. 
The existence of the section follows from Observation \ref{obs_blowup_has_a_section} (the section is the inclusion of the exceptional divisor).
\end{proof}

Recall now that there is a well-known isomorphism $\calP(4,6)\cong \overline{\calM}_{1,1}$ over $\mathbb{Z}[\frac{1}{6}]$. Indeed, using the change of variables in \cite[Section III.1]{Silverman}, we can show that every stable curve in Weierstrass equation over $\mathbb{Z}[\frac{1}{6}]$ can be written as $y^2z = x^3 + Axz^2+Bz^3\subseteq \bP^2$. Then proceeding as in \cite[\href{https://stacks.math.columbia.edu/tag/072S}{Tag 072S}]{stacks-project}, we consider the schemes $W =\Spec( \mathbb{Z}[A,B,\frac{1}{6}])$ and $$E_W:=\{y^2z = x^3 + Axz^2+Bz^3\}\subseteq \mathbb{P}^2_W$$
and the action of $\bG_m$ on $W$ induced by the action of $\bG_m$ on $\mathbb{P}^2_{\mathbb{Z}}$ with weight 2 (resp. 3, 0) on the coordinate $x$ (resp. $y$, $z$). This induces an action of $\bG_m$ on $W$ with weight $4$ on $A$
 and $6$ on $B$. Proceeding as in \emph{loc. cit.} one can show the desired isomorphism $[W\smallsetminus \{0\}/\bG_m]\cong \overline{\calM}_{1,1}$.
 
 The following Theorem identifies $B_Z^{(4,6)}\calP(2,3,4)$ with the universal family of 1-pointed genus 1 curves over $\overline{\calM}_{1,1}$.
\begin{theorem}\label{theorem_iso_weighted_blowup_M12}
There are isomorphisms $B_Z^{(4,6)}\calP(2,3,4) \cong \overline{\calM}_{1,2}$ and $\calP(4,6) \cong \overline{\calM}_{1,1}$ which make the following diagram cartesian, where the morphism $\phi$ is the universal family:
$$\xymatrix{B_Z^{(4,6)}\calP(2,3,4) \ar[d]_{\pi}\ar[r] &\overline{\calM}_{1,2} \ar[d]^{\phi}\\ \calP(4,6) \ar[r] &  \overline{\calM}_{1,1}.}$$
\end{theorem}
\begin{proof}
From Proposition \ref{prop_family_of_ellpitic}, $\pi$ is a family of genus 1 curves with a section, so it induces a morphism $\psi:\calP(4,6) \to \overline{\calM}_{1,1}$ and a \emph{cartesian} diagram like the one above. Therefore for proving the statement, it suffices to check that $\psi$ is an isomorphism.

We check that $\psi$ is bijective on geometric points and representable (therefore it is finite). 
 We can check that it is bijective on the level of coarse spaces, and since the source is proper and both coarse spaces are isomorphic to $\bP^1$, it suffices to check injectivity.
In other terms, it suffices to check that two distinct fibers of $\pi$ have different $j$-invariant. This follows from the explicit descriptions of the fibers of $\pi$ of Proposition \ref{prop_family_of_ellpitic}.
To check it is representable, we use Lemma \ref{lemma_representable_automorphisms_fibred_product}. It suffices to check that for every point $p\in \overline{\calM}_{1,1}$, there is a point $q \in B_Z^{(4,6)}\calP(2,3,4)$ which maps to it and which has no automorphisms: for every pair of $(\beta_4,\beta_6)$ it suffices to choose a point $(\alpha_2,\alpha_3,\alpha_4)$ with $\alpha_2\alpha_3 \neq 0$, $\alpha_4 = \beta_4$ and $\alpha_3^2 = \alpha_2^3 + \beta_4\alpha_2 +\beta_6$.

Observe now that $\psi$ is an isomorphism over $\Spec(\overline{\mathbb{Q}})$. Indeed the restriction $\psi_{\overline{\mathbb{Q}}} : \calP(4,6)_{\overline{\mathbb{Q}}} \to \overline{\calM}_{1,1, \overline{\mathbb{Q}}}$ is still bijective and representable. In particular, for every point $p\in \calP(4,6)(\Spec(\overline{\bQ}))$, the induced map $\operatorname{Aut}_p(\calP(4,6)_{\overline{\mathbb{Q}}})\to \operatorname{Aut}_{\psi_{\overline{\bQ}}(p)}(\overline{\calM}_{1,1,\overline{\mathbb{Q}}})$ is injective. But an injective map between finite groups of the same cardinality is an isomorphism,
so $\psi_{\overline{\bQ}}$ is an isomorphism on the groupoids of $\overline{\mathbb{Q}}$-points. Then from Zariski's main theorem (see \cite[Theorem A.5]{AI19}) it is an isomorphism. But $\Spec(\overline{\mathbb{Q}})\to \Spec(\bQ)$ is an fpqc morphism, so if $\psi_{\overline{\mathbb{Q}}}$ is an isomorphism also $\psi_{\bQ}$ was an isomorphism as being an isomorphism is local in the fpqc topology (see \cite[\href{https://stacks.math.columbia.edu/tag/02YJ}{Tag 02YJ}]{stacks-project}).

We observe now that $\psi_*\calO_{\calP(4,6)}$ is a line bundle. Indeed, $\psi$ is flat from "miracle flatness Theorem" (see  \cite[\href{https://stacks.math.columbia.edu/tag/00R3}{Tag 00R3}]{stacks-project}), as it is a map between two regular 2-dimensional stacks, with 0-dimensional fibers. Therefore $\psi_*\calO_{\calP(4,6)}$ is a vector bundle, since $\overline{\calM}_{1,1}$ is regular, and flat modules over a regular local ring are free. But since $\psi_\bQ$ is an isomorphism, $\psi_*\calO_{\calP(4,6)}$ has rank 1 over the generic point. Since $\overline{\calM}_{1,1}$ is connected, $\psi_*\calO_{\calP(4,6)}$ has rank 1.

Recall finally that $\psi$ is finite, so to check that it is an isomorphism it suffices to check that
$\psi^{\#} :\calO_{\overline{\calM}_{1,1}} \to \psi_*\calO_{\calP(4,6)}$ is an isomorphism. It suffices to check the statement \'etale locally on the target, where we have a morphism $A_p\to B_p$ of noetherian local rings, which sends $1\mapsto 1$ and such that $B_p$ has rank 1 as an $A_p$-module. But then if we denote by $m_A$ (resp. $m_B$) the maximal ideal of $A_p$ (resp. $B_p$), then $A_p/m_A\to B_p/m_B$ is surjective. So from Nakayama, $A_p \to B_p$ is surjective. It is injective since both rings are integral and $\psi$ is an isomorphism on the generic fiber, so it is an isomorphism.\end{proof}

\begin{lemma}\label{lemma_representable_automorphisms_fibred_product}
Let $f:\calX\to \calY$ and $g:\calZ \to \calY$ be morphisms of algebraic stacks, and let $F:= \calX\times_{\calY}\calZ$ be its fibered product. Assume that for every geometric point $p\in \calX$, there is a point $q \in F$ which maps to $p$ and such that $\operatorname{Aut}_F(q) = \{\operatorname{Id}\}$. Then $f$ is representable.
\end{lemma}
\begin{proof}
Recall that there is an equivalence of groupoids 
$$|F(\Spec(k))| \cong \{(a,b,\sigma): a \in \calX(\Spec(k)), b\in \calZ(\Spec(k)), \sigma \in \operatorname{Hom}_{\calY}(f(a),g(b)) \}.$$
 The automorphisms of a point $(a,b,\sigma)$ are pairs $(\mu,\nu)\in \operatorname{Aut}_{\calX}(a) \times \operatorname{Aut}_{\calZ}(b)$ which make this diagram commutative:
$$\xymatrix{ f(a) \ar[d]_{f(\mu)} \ar[r]^{\sigma} & g(b)\ar[d]^{g(\nu)} \\
f(a) \ar[r]^{\sigma} & g(b)}$$
In particular, for every $\mu$ such that $f(\mu) = \operatorname{Id}$, the pair $(\mu,\operatorname{Id})$ is an automorphism. Therefore if $\operatorname{Aut}_F(a,b,\sigma) = \operatorname{Id}$, the map $f$ induces an injection $\operatorname{Aut}_{\calX}(a) \to \operatorname{Aut}_{\calY}(f(a))$.

Now, our assumptions guarantee that for every geometric point $p\in \calX$ there is a point $q \in F$ which maps to $p$ and which has no automorphisms, so for every geometric point $p \in \calX$ the map $\operatorname{Aut}_{\calX}(p) \to \operatorname{Aut}_{\calY}(f(p))$ is injective. The result follows from \cite[Lemma 2.3.9]{AH}.\end{proof}
\begin{observation}
Over a field, one can also prove that $B_Z^{(4,6)}\calP(2,3,4) \cong \overline{\calM}_{1,2}$ as follows. First, observe that both spaces have the same coarse moduli space (proved in \cite[Theorem 2.3]{Mas14}, observe that the coarse moduli spaces of $\calP(2,3,4)$ and $\calP(1,2,3)$ agree). Then it suffice to study the automorphisms groups of the points of codimension one. Indeed since $\overline{\calM}_{1,2}$ is smooth and generically a scheme, the stabilizers on the points of codimension one are enough to characterize the stack structure of $\overline{\calM}_{1,2}$ (see  \cite[Theorem 1]{Ger17}). 
\end{observation}
\section{The Brauer group of $\overline{\calM}_{1,2}$}
Let $S$ be a scheme where $6$ is invertible. The goal of this s ection is to prove the following:
\begin{theorem}\label{teo_brauer}
The pull-back map $\operatorname{Br}'(S)\to \operatorname{Br}'(\overline{\calM}_{1,2;S})$ is an isomorphism.\end{theorem}
The proof is divided into several steps, but the key ingredient is Minseon Shin's result on the cohomological Brauer group of weighed projective stacks:
\begin{theorem}[\cite{Minseon}]Let $\calP_S$ be a weighted projective stack over $S$. Then the pull-back map
$ \operatorname{Br}'(S)\to \operatorname{Br}'(\calP_S)$ is an isomorphism.
\end{theorem}
Up to working a connected component at the time, we can assume that $S$ is connected. 
\begin{notation}We will adopt the following notations:
\begin{enumerate}
    \item $\calM:=\overline{\calM}_{1,2;S}$;
    \item $\calP:=\calP(2,3,4)_S$;
    \item $E:= \calP(4,6)_S$, and
    \item $f:\calM\to \calP$ the blow-down.
\end{enumerate}\end{notation}

Before proceeding with the proof of Theorem \ref{teo_brauer}, it is useful to recall the following \begin{observation}\label{oss_trivial_gerbe_is_1_twisted}
A $\bG_m$-gerbe $\sG \to \sY$
over a stack $\sY$ is trivial if and only if it has representable morphism $\pi:\sG\to \sY\times B\bG_m $ over $\sY$. This in turn is equivalent to the data of a line bundle $\calL$ on $\sG$ with the following property. For every point $p\in \sG$, if we denote by $\sK_p$ be the kernel of $\operatorname{Aut}_p(\sG) \xrightarrow{\pi} \operatorname{Aut}_p(\sY)$, then we require $\sK_p$ to act faithfully on $\calL_p$. Such a line bundle is called $1$-twisted or $(-1)$-twisted line bundle (see \cite[Lemma 3.1.1.8]{Lieb08}).
\end{observation}

\begin{bf}Step 1:\end{bf} It suffices to prove that $\oH^1(\calP, R^1f_*\bG_m)^{tor} = \oH^0(\calP, R^2f_*\bG_m)^{tor}=0$ and that the morphism $\oH^1(\calM, \bG_m)\to \oH^0(\calP, R^1f_*\bG_m)$ is surjective.

Observe that we can compute $R\Gamma(\calM,\cdot)$ by composing the functors $Rf_*$ with $R\Gamma(\calP,\cdot)$, where $R\Gamma$ is the derived functor of global sections. Then we can consider the Grothendieck spectral sequence associated with this composition of functors, and its seven term exact sequence:

\begin{align*}
\oH^1(\calP, \bG_m) \to &\oH^1(\calM, \bG_m) \xrightarrow{a} \oH^0(\calP, Rf_*^1\bG_m) \to \\\to \oH^2(\calP, \bG_m)\to &\operatorname{Ker}(\oH^2(\calM, \bG_m)\to \oH^0(\calP, Rf_*^2\bG_m))\to\oH^1(\calP, Rf_*^1\bG_m).
\end{align*}
In particular, it suffices to prove that $a$ is surjective and that $\oH^1(\calP, R^1f_*\bG_m)^{tor} = \oH^0(\calP, R^2f_*\bG_m)^{tor} = 0$ to guarantee an isomorphism $\oH^2(\calP, \bG_m)^{\text{tor}} \to \oH^2(\calM, \bG_m)^{\text{tor}} $.

\begin{bf}Step 2:\end{bf} If $\iota:Z\to \calP$ is the inclusion of the locus where $\calM\to \calP$ is not an isomorphism, then $R^1f_*\bG_m = \iota_* \mathbb{Z}$. Moreover, for every \'etale morphism $U\to \calP$, we can identify $R^1f_*\bG_m(U)$ with the relative Picard group of $U\times_\calP \calM\to U$. With this identification, the isomorphism $\iota_*\mathbb{Z}\to R^1f_*\bG_m$ sends $1$ to the ideal sheaf of the exceptional divisor. 
\begin{notation} For every geometric point $x\in Z$,
let $A=\calO_{\calP,x}^{\operatorname{sh}}$ with maximal ideal $m$, let $\calB:=\calM\times_{\calP}\Spec(A)$ with coarse moduli space $B$, let $\calZ\subseteq \Spec(A)$ be the pull-back $Z\times_\calP \Spec(A)$ and let $p:\calB\to \Spec(A)$ be the second projection. We summarize the new notations in the following cartesian diagrams:
$$\xymatrix{\calB \ar[d]\ar[r] & \calM \ar[d] \\ \Spec(A)\ar[r] & \calP}\text{ } \text{ } \xymatrix{\calZ \ar[d]\ar[r] & Z \ar[d] \\ \Spec(A)\ar[r] & \calP}$$
\end{notation}

 We can compute  $(R^1f_*\bG_m)_x$ by replacing $\calP$ with $\Spec(A)$. Thus as $A$ is strictly henselian, it suffices to check that $\oH^1(\calB,\bG_m)\cong \mathbb{Z}$. To do so, first observe that $\Pic(\Spec(A))=\{0\}$ as $A$ is a regular local ring. Moreover as $\calZ$ has codimension 2 in $\Spec(A)$, $$\Pic(\Spec(A)\smallsetminus \calZ) = \Pic(\Spec(A)) = \{0\}.$$
Observe now that if $E_{\calZ}\cong \calP(4,6)_{\calZ}$ is the exceptional divisor of $\calB$, the complement of $E_{\calZ}$ in $\calB$ is isomorphic to $\Spec(A)\smallsetminus {\calZ}$. Therefore we have an exact sequence
$$\mathbb{Z}[E_{\calZ}]\xrightarrow{j} \Pic(\calB)\to \Pic(\Spec(A)\smallsetminus {\calZ})\to 0.$$
Therefore $\Pic(\calB)$ is cyclic, so it suffices to check that $\Pic(\calB)$ is not finite. It is enough then to check that the image of the pull-back $\Pic(\calB)\to \Pic(E_{\calZ})$ contains the class of $\calO_{\calP(4,6)_{\calZ}}(1)$.
But this is true: the pull-back map $\Pic(\calP(4,6)_S)\to \Pic(\calP(4,6)_{\calZ})$ (which sends $\calO_{\calP(4,6)_S}(1)\mapsto \calO_{\calP(4,6)_{\calZ}}(1)$) can be factored as 
$$E_{\calZ}\cong \calP(4,6)_{\calZ}\to \calB\to \calM\to \calP(4,6)_S.$$

\begin{bf}
Step 3:
\end{bf} The morphism $a:\oH^1(\calM,\bG_m) \to \oH^0(\calP,R^1f_*\bG_m)$ is surjective.

Indeed, $a$ is the restriction $\Pic(\calM)\to \Pic(E)$, which is surjective from Observation \ref{Oss_pull_back_exceptional_gives_iso_on_pic} (the pull-back of the ideal sheaf of the exceptional divisor is $\calO_E(1)$).

\begin{bf} Step 4:\end{bf} $\oH^1(\calP, R^1f_*\bG_m)^{tor} = 0$.

Indeed, if we denote by $\iota:Z\to \calP$ the inclusion of $Z$, then $R^1f_*\bG_m \cong \iota_*\mathbb{Z}$, where the isomorphism $\iota_*\mathbb{Z}\to R^1f_*\bG_m$ sends $1$ to the ideal sheaf of the exceptional divisor. So $$\oH^1(\calP, R^1f_*\bG_m) = \oH^1(\calP, \iota_*\mathbb{Z}) = \oH^1(Z, \mathbb{Z}).$$
The latter has no torsion from \cite[Corollary 7.9.1]{Wei91}.

\begin{bf}
Step 5:
\end{bf}
In the following two steps we show that for every geometric point $x\in \calP$, the stalk $(R^2f_*\bG_m)_p$ has no torsion (so also $\oH^0(\calP,R^2f_*\bG_m)$ has no torsion). This concludes the proof.

 Our goal is to show that $\oH^2(\calB,\bG_m)$ has no torsion.
Therefore, let $\calX\to \calB$ be a $\bG_m$-gerbe corresponding to a torsion class in cohomology. Let $\hat{A}$ be the completion of $A$ along the ideal of $\calZ$, and let $\hat{\calB}:=\calB\times_{\Spec(A)}\Spec(\hat{A})$ and similarly $\hat{\calX}:=\calX\times_{\Spec(A)}\Spec(\hat{A})$.

\begin{bf}Step 6:\end{bf} The gerbe $\hat{\calX}\to \hat{\calB}$ has a section (so it is trivial).

First, observe that from \cite[Theorem 1.4]{Ols05}, we have $\varprojlim Coh(\calB_n) = Coh(\hat{\calB})$, where we denote by $\calB_n=\calB\times_{\Spec(A)}\Spec(A/m^n)$. Then from Tannaka duality
\begin{align*}
    \Hom_{\hat{\calB}}(\hat{\calB}, \hat{\calX}) &= \Hom_{Coh(\hat{\calB})}(Coh(\hat{\calX}),Coh(\hat{\calB})) = \Hom_{Coh(\hat{\calB})}(Coh(\hat{\calX}),\varprojlim Coh(\calB_n)) \\&=\varprojlim \Hom_{Coh(\hat{\calB})}(Coh(\hat{\calX}), Coh(\calB_n)) = \varprojlim \Hom_{\hat{\calB}}(\calB_n,\hat{\calX}).
\end{align*}
In particular, we have to show that:
\begin{enumerate}
    \item If we can lift $\calB_n\to \hat{\calX}$, then we can also lift $\calB_{n+1}\to \hat{\calX}$, and
    \item The map $\calB_0\to \hat{\calB}$ can be lifted to $\calB_0\to \hat{\calX}$.
\end{enumerate}
We begin by (1). Since lifting $\calB_{n+1}\to \calX$ is equivalent to the gerbe $\calX\times_{\hat{\calB}}\calB_{n+1} \to \calB_{n+1}$ having a section, it suffices to prove that if $\calX\times_{\hat{\calB}}\calB_{n} \to \calB_{n}$ is trivial then also $\calX\times_{\hat{\calB}}\calB_{n+1} \to \calB_{n+1}$ is trivial. 
Thus from Observation \ref{oss_trivial_gerbe_is_1_twisted}, we need to show that we can lift a $1$-twisted line bundle from $\calB_n\hookrightarrow \calB_{n+1}$. 

To check that a line bundle $\calL$ on $\calX\times_{\hat{\calB}}\calB_{n+1}$ is $1$-twisted one needs to check the condition of Observation \ref{oss_trivial_gerbe_is_1_twisted} only at the geometric points. Since the geometric points of $\calX\times_{\hat{\calB}}\calB_{n+1}$ are the same as those of $\calX\times_{\hat{\calB}}\calB_{n}$, any extension of a $1$-twisted line bundle from $\calX\times_{\hat{\calB}}\calB_{n}$ to $\calX\times_{\hat{\calB}}\calB_{n+1}$ will be $1$-twisted. Therefore we can ignore the condition of being 1-twisted, and our goal is to show that we can extend a line bundle $\calL$ from $\calX\times_{\hat{\calB}}\calB_n$ to $\calX\times_{\hat{\calB}}\calB_{n+1}$.

Let $I$ the ideal sheaf of $\Spec(A_n)\to \Spec(A_{n+1})$, let $\calX_0 = \calX\times_{\Spec(A)}\calZ$ and let $\xi:\calX_0\to \calZ$ be the composition of $\calX_0 \xrightarrow{\nu} \calB_0 \xrightarrow{\mu} \calZ$. Then the obstruction to extend $\calL$ lies in $\oH^2(\calX_0,\xi^*I)$. We want to show that $\oH^2(\calX_0,\xi^*I)=0$.
Indeed, since $\calX\to \calB\to B$ is a good moduli space, also $\nu:\calX_0\to B_0 = B\times_{\Spec(A)}\calZ$ is a good moduli space, as good moduli spaces are stable under base change (see \cite[Proposition 4.7]{GMS_jarod}). 
Therefore $\nu_*\calO_{\calX_0} = \calO_{B_0}$ and $\nu_*$ is exact as good moduli spaces are cohomologically affine. Then we have 
$$\oH^2(\calX_0,\xi^*I)=\oH^2(\calX_0,\nu^*\mu^*I) =\oH^2(B_0,\nu_*\nu^*\mu^*I) =\oH^2(B_0,\mu^*I).$$
We aim at showing that $\oH^2(B_0,\mu^*I)=0$. Observe that can compute the derived functor $R\Gamma(B_0,\cdot)$ as a composition of $R\mu_*$ and $R\Gamma(\calZ,\cdot)$. But the latter is an exact functor, since  $\calZ$ is affine (as it is closed in $\Spec(A)$). Therefore $\oH^2(B_0,\mu^*I) = \oH^0(\calZ, R^2\mu_*I)$, and $R^2\mu_*I=0$ since $B_0\to \calZ$ has relative dimension 1. Therefore $\oH^0(\calZ, R^2\mu_*I) = 0$, so $\oH^2(B_0,\mu^*I)=0$ as desired.

To show (2), observe that there is a surjective closed embedding $\iota:\calP(4,6)_{\calZ}\hookrightarrow \calB_0$ (i.e. $\calB_0^{red} = \calP(4,6)_{\calZ}$). In particular, $\iota_*\bG_m = \bG_m$ and $R^i\iota_* = 0$ for $i>0$ (as a quotient of a strictly Henselian ring is strictly Henselian). Thus $\oH^2(\calB_0, \bG_m)^{tor} = \oH^2(\calP(4,6)_{\calZ}, \bG_m)^{tor} = 0$ where the last equality follows from \cite{Minseon}, the fact that $\calZ$ is closed in $\Spec(A)$, and $A$ is strictly Henselian.

\begin{bf}Step 7:\end{bf}
We show that $\calX\to \calB_0$ is trivial.

Let $(\sS ch/\Spec(A))^{op}$ be the opposite category of schemes over $\Spec(A)$ and consider the following three functors.
First, for every scheme $T\to \Spec(A)$, consider the functors of sets of the groupoids $\calH om_{\Spec(A)}(\calB_T, \calX_T)$ and $\calH om_{\Spec(A)}(\calB_T, \calB_T)$. We will denote these functors as $F_1$ and $F_2$. Observe that the map $\calX_T\to \calB_T$ induces a natural transformation $F_1\to F_2$. Lastly consider the constant functor, which we denote by $F_3$, that sends $T\to \Spec(A)$ to a set with a single element which we will denote by $\{\operatorname{Id}\}$. There is a natural transformation $F_3\to F_2$ such that for every map $T\to \Spec(A)$ sends $\operatorname{Id} \in F_3(T\to \Spec(A))$ to the identity in $\calH om_{\Spec(A)}(\calB_T, \calB_T)$. We will denote by $F$ the fiber product $F_1\times_{F_2}F_3$. This is the functor that sends $T\to \Spec(A)$ to the set of sections of $\calX_T\to \calB_T$.

We first show that $F$ commutes with limits. A limit in $(\sS ch/\Spec(A))^{op}$ is a colimit in $(\sS ch/\Spec(A))$, so let $\varinjlim T_i$ be such a colimit. Observe now that, since $\calX$ and $\calB$ are of finite type, and since $\calH om_{\Spec(A)}(\calB_T, \calX_T) = \calH om_{\Spec(A)}(\calB_T, \calX)$ and $\calH om_{\Spec(A)}(\calB_T, \calB_T) =\calH om_{\Spec(A)}(\calB_T, \calB)$,
\begin{center}$F_1(\varinjlim T_i) = \varinjlim F_1(T_i)$ and similarly $F_2(\varinjlim T_i) = \varinjlim F_2(T_i)$. \end{center}Similarly, also $F_3$ commutes with colimits. Notice also that $F$ is a finite limit of the $F_i$, as fiber products are limits. Now the desired commutativity follows from this string of equalities:
$$\varinjlim_j F(T_j) = \varinjlim_j (\varprojlim_i F_i)(T_j) = \varinjlim_j \varprojlim_i F_i(T_j) =  \varprojlim_i \varinjlim_jF_i(T_j) = \varprojlim_i F_i(\varinjlim_j T_j) = F(\varinjlim_j T_j). $$

The first and last equality follow from the definition of $F$, the second equality follows from the definition of limits in sets, the third equality follows from \cite[\href{https://stacks.math.columbia.edu/tag/04AX}{Section 04AX}]{stacks-project}, and the fourth one from the fact that $F_i$ commute with colimits.
Therefore we can apply \cite[Theorem 3.4]{AHR19}, and from the results of the previous step the gerbe $\calX\to \calB$ has a section.

\section{Chow ring of $\overline{\calM}_{1,2}$ and $\calM_{1,2}$}
In this section we will assume that $S $ is a field, which we denote by $\Spec(k)$. As usual, we assume that $6$ is a unit in $k$.
\subsection{Equivariant intersection theory}\label{subsection facts on equivariant chow}In this subsection we recall the relevant definitions that we will need on equivariant intersection theory and we refer the \cite{EG98} for more details. It is useful to keep in mind the following three facts:
\begin{enumerate}
\item If $V$ is a vector bundle on a scheme $X$, then the pull-back map $A^*(X)\to A^*(V)$ is an isomorphism,
\item If $C\subseteq X$ is a closed subset of a scheme $X$, then there is an exact sequence $$A_*(C)\to A_*(X) \to A_*(X\smallsetminus C) \to 0 \text{, and}$$
\item For every $i$ and avery $n>i$ there is a linear action of $G$ on a vector space $V$ with an open subset of codimension at least $n-i$ where $G$ acts freely.
\end{enumerate}

In particular, if $X$ is a scheme, $\pi:V\to X$ a vector bundle of rank $n$,  and $C\subseteq X$ a closed subset of codimension $i+1$, $A_{\ge i}(X) \cong A_{\ge n+ i}(\pi^{-1}(X\smallsetminus C))$.

Assume now that a linear algebraic group $G$ of dimension $g$ acts on $X$, and let $n$ be the dimension of $X$.
One can consider the quotient stack $[X/G]$ and coupling the paragraph above with point (3), for every $n\in \mathbb{Z}$,  it
is natural to define its Chow groups $A_n([X/G])$ as follows.  First, consider a representation $V$ of $G$ of dimension $l$, such that $G$ acts freely on an open subset $U\subseteq V$ of codimension at least $n-i$.  We define $A_i([X/G]) := A_{i + l -g}([X\times U /G])$.
The advantage is that $[X\times U /G]$ is an algebraic space (or a scheme if for example $X$ is projective, see \cite[Proposition 23]{EG98}), so its Chow groups are defined classically. It follows from Bogomolov's double filtration argument that $A_i([X/G])$ are well defined (i.e. the definition does not depend on $V$ or $U$). Moreover, the definition of $A_n([X/G])$ does not depend on the presentation of $[X/G]$: if $[X/G]\cong [Y/H]$ for an algebraic space $Y$ and a group $H$, then $A_n([X/G]) \cong A_n([Y/H])$.

Equivariant Chow groups enjoy some properties of the usual Chow groups. We list now the properties we will use in the rest of the manuscript (and we refer to \cite{EG98} for further details).

Let $f:X\to Y$ be a proper morphism and $g:Z\to Y$ be a flat one. Assume that we have a linear algebraic group $G$ acting on $X,Y$ and $Z$ in a way such that $f$ and $g$ are equivariant. Then we have a proper push-forward $f_*:A_*([X/G])\to A_*([Y/G])$ and a flat pull-back $g^*:A^*([Y/G]) \to A^*([Z/G])$. Both proper push-forward and flat pull-backs are functorial, and the following diagram commutes, where $d$ is the dimension of the fibers of $g$:

\begin{equation}
\label{comm diag}
    \xymatrix{A_{* + d}([X\times_Y Z /G])\ar[d] & A_*([X /G]) \ar[d]^{\text{proper}} \ar[l] \\ A_{* + d}([Z /G]) & A_*([Y /G]) \ar[l]_-{\text{flat}}.}
\end{equation}
In particular, if $U$ is an open subset of a vector bundle as above, then $U\times Y\to Y$ is flat, and we can use the diagram above to explicitly control proper push-forwards on the level of Chow groups.

\begin{remark}\label{r}
For any representable morphism $\calZ\to [X/G]$, the fiber product $Z:=\calX\times_{[X/G]}X$ is an algebraic space with an action of $G$. Moreover, $\calZ = [Z/G]$, the second projection
$Z\to X$ is equivariant, and $\calZ\to [X/G]$ is induced by $Z\to X$. In particular, we can define push-forwards and the pull-backs as above for representable morphisms where the target is a quotient stack.
\end{remark} 

Finally we will also use that if $U\subseteq X$ is an open subset which is invariant for the action of $G$ and with complement $C$, then there is an exact sequence
$$A_*([C/G]) \to A_*([X/G]) \to A_*([U/G])\to 0.$$

\begin{prop}\label{prop_seminorm_is_a_chow_envelope}
Let $f:X^{sn}\to X$ be the seminormalization of $X$ . Then $f$ is a Chow envelope, namely $f_* : A_*(X^{sn})\to A_*(X)$ is surjective.
\end{prop}
We refer the reader to \cite[\href{https://stacks.math.columbia.edu/tag/0EUK}{Tag 0EUK}]{stacks-project} for the definition and properties of the seminormalization.
\begin{proof}
It is clear that $f_*[X^{sn}] = [X]$. Then for any irreducible closed subset $i:V\hookrightarrow X$, $[V] = i_*[V] = i_*\circ \nu_* [ V^{sn}]$, where $\nu:V^{sn}\to V$ is the seminormalization. But the seminormalization is functorial, so $V^{sn}\to X$ lifts to $j:V^{sn}\to X^{sn}$. Then $[V] = i_*\circ \nu_* [ V^{sn}] = f_*\circ j_* [ V^{sn}]$.
\end{proof}
\begin{cor}\label{cor_seminorm_is_chow_envelope_for_stacks}
If $\calX=[Y/G]$ is an algebraic stack and $Y^{sn}$ is the seminormalization of $Y$, then $G$ acts on $Y^{sn}$, the map $[Y^{sn}/G] \to [Y/G]$ is a seminormalization, and it is a Chow envelope.
\end{cor}
\begin{proof}
As noted before Remark \ref{r}, it suffices to prove the desired statement if we replace $Y$ (resp. $Y^{sn}$) with $Y\times U$ (resp. $Y^{sn}\times U$) for $U$ an open subset of a $G$-vector space as at the beginning of the subsection. Then the desired statement follows from Proposition \ref{prop_seminorm_is_a_chow_envelope}, as $Y^{sn}\times U = (Y\times U)^{sn}$.
\end{proof}
\subsection{Pull-backs and push-forwards for maps between vector bundles}
We prove three lemmas about morphisms between vector bundles that will be useful later.
\begin{lemma}\label{lemma_pull_back_class_vb}
Let $\phi: V\cong V_1\oplus V_2 \to B$ be a vector bundle which is the sum of two vector bundles $V_1$ and $V_2$, and let $n$ be the rank of $V_1$. Consider the projection $p:V\to V_1$ over $V_1$. Then $p$ is flat and if we denote by $[B]\in A_*(V_1)$ the class of the zero section, then $p^*([B]) = \phi^*(c_n(V_1))$, where $c_n$ is the $n$-th Chern class of $ V_1$.
\end{lemma}
\begin{proof}Let $\phi_1:V_1\to B$ the projection. 
From \cite[Proposition 1.9, Theorem 14.1]{Ful13} the map $\phi_1^*:A^*(B)\to A^*(V_1)$ is an isomorphism sending $c_n(V_1)\mapsto [B]$. The map $p$ is flat, since we can identify it with the first projection $V_1\times_B V_2\to V_1$. The claim follows since $\phi^* = p^*\circ \phi_1^*$.
\end{proof}
\begin{observation}\label{Oss_push_forward_power_line_bundle}
Let $V\to X$ be a line bundle. For every $n>0$ we have the $n-$th power morphisms $p_n: V\to V^{\otimes n}$. This is a finite morphism which induces a map $(p_n)_*:A_*(V) \to A_*(V^{\otimes n})$ such that $(p_n)_*([V]) = n[V^{\otimes n}]$.
\end{observation}
\begin{lemma}\label{lemma_pushforward_O(-1)_chow}
Let $X$ be a scheme, let $\calL$ be a line bundle on $X$, and let $(a_1,...,a_n)$ be positive integers. Assume that $n$ is invertible on $X$. Consider the morphism $$\bigoplus_{i=1}^n p_{a_i} : \calL \to \bigoplus_{i=1}^n \calL^{\otimes a_i}$$ where $p_{a_i}$ is the $a_i$-th power morphism of
Observation \ref{Oss_push_forward_power_line_bundle}, and let us denote this morphism by $\phi$. If we denote by $\calW:=\bigoplus_{i=1}^n \calL^{\otimes a_i}$ and by $\pi:\calW \to X$ the projection, then $\pi^*$ induces an isomorphism $A^*(X) \to A^*(\calW)$ which identifies $\prod a_i c_1(\calL)^{n-1}$ with $\phi_*([\calL])$.
\end{lemma}
\begin{proof}
We can factor $\phi$ as the composition of the diagonal morphism $g:\calL\to \calL^{\oplus n}$ and $h:\calL^{\oplus n} \to \calW$, where
$h$ on the $i$-th direct summand is $p_i$. 
The map $g$ fits in an exact sequence
$$0\to \calL\xrightarrow{g} \calL^{\oplus n}\xrightarrow{\psi} \calL^{\oplus n-1} \to 0 $$ where $\psi((a_1,...,a_n)) = (a_2-a_1,...,a_n - a_1)$. Observe that $g$ splits as $n$ is invertible on $X$. Then from Lemma \ref{lemma_pull_back_class_vb}, $\psi$ is flat and if we identify with $X_0\subseteq \calL^{\oplus n-1}$ the zero section and with $q:\calL^{\oplus n}\to X $ the projection, then $$g_*[\calL] = [g(\calL)] = [\psi^{-1}(X_0)] = \psi^*([X_0]) = q^*(c_1(\calL)^{n-1}).$$

Now, $\phi_*([\calL]) = h_*(g_*[\calL]) = h_*(q^*(c_1(\calL)^{n-1})).$ But $\pi\circ h = q$, therefore if we denote by $\alpha =c_1(\calL)^{n-1} $ we have  $$h_*(q^*(\alpha)) = h_*(h^*\pi^*(\alpha)) = h_*(h^*\pi^*(\alpha)\cap [\calL^{\oplus n}]) = \pi^*(\alpha)\cap h_*([\calL^{\oplus n}]) = \left(\prod a_i\right)\pi^*(\alpha)$$
where the second equality follows since $\cap[\calL^{\oplus n}]$ is the identity in the Chow ring of $\calL^{\oplus n}$, and the last one from Observation \ref{Oss_push_forward_power_line_bundle}.
\end{proof}
\begin{cor}\label{corollary_map_vector_bundles_in_chow}
We adopt the notations of Lemma \ref{lemma_pushforward_O(-1)_chow}. Let $V$ be a vector bundle of rank $m$ over $X$ and let $i:\calW \to \calW\oplus V$ be the inclusion. Then if $q:\calW\oplus V\to X$ is the projection, $q^*$ induces an isomorphism $A^*(X) \to A^*(\calW\oplus V)$ which identifies $\left(\prod a_i \right)c_m(V)c_1(\calL)^{n-1}$ with $(i\circ \phi)_*([\calL])$.
\end{cor}
\begin{proof} Let $\pi:\calW\to X$ the projection. From Lemma \ref{lemma_pushforward_O(-1)_chow} we have $(i\circ \phi)_*([\calL]) = \left(\prod a_i\right)i_*(\pi^*c_1(\calL)^{n-1})$. But $\pi = q\circ i$ and from Lemma \ref{lemma_pull_back_class_vb}, $i_*[\calW] = q^*(c_m(V))$. So$$ i_*(\pi^*c_1(\calL)^{n-1}) = i_*(i^*(q^*c_1(\calL)^{n-1})\cap [\calW]) =  q^*c_1(\calL)^{n-1}\cap i_*([\calW]) = q^*(c_1(\calL)^{n-1}c_m(V)). $$\end{proof}
\subsection{Preliminaries for $A^*(\overline{\calM}_{1,2})$} We prove some auxiliary results that will be needed for computing $A^*(\overline{\calM}_{1,2})$ and $A^*(\calM_{1,2})$.

\begin{notation}We adopt the following notation:
\begin{enumerate}
    \item $\calM:=\overline{\calM}_{1,2}$, $\calP:= \calP(2,3,4)$,
    and $\calU:=\calP\smallsetminus Z$,
    \item $f:\calM \to \calP$ will be the weighted blow-up of $\calP$ at $Z$ and $\pi:\calM \to \overline{\calM}_{1,1}$ the family of genus one and one pointed stable curves.
\end{enumerate}
\end{notation}

\begin{lemma}\label{lemma_chow_ring_weighted_proj_space}
Given $\{a_i\}_{i=1}^n \subseteq\mathbb{Z}$, we have $A^*(\calP(a_1,...,a_n)) = \mathbb{Z}[t]/a_1\cdot...\cdot a_nt^n$ where $t=c_1(\calO_{\calP(a_1,...,a_n)}(1))$.
\end{lemma}
\begin{proof}
As $\calP(a_1,...,a_n)$ is the complement of the zero section of the vector bundle $[\mathbb{A}^n/\bG_m]\to B\bG_m$, where $\mathbb{G}_m$ acts on $\mathbb{A}^n$ with weights $a_1,...,a_n$. Therfore from \cite[Lemma 2.2]{MV06} we have $A^*(\calP(a_1,...,a_n)) = A^*(B\bG_m)/c_n(\mathbb{A}^n) = \mathbb{Z}[t]/a_1\cdot...\cdot a_nt^n$. 
\end{proof}
We will use the following proposition to determine the ring structure on $A^*(\overline{\calM}_{1,2})$.
\begin{prop}\label{prop_inclusion_of_a_pt_with_some_coordinates_0_in_weighted_proj_stack}Consider $\calX = \calP(b, a_1,...,a_n)$ a weighted projective stack. If $\alpha:B\bmu_{b}\to \calX$ is the inclusion of  $[1,0,...,0]\cong B\bmu_{b}$, then:
\begin{enumerate}
    \item $\alpha_*[B\bmu_{b}] = c_1(\calO_\calX(a_1))\cdot ... \cdot c_1(\calO_\calX(a_n))$ in $A_*(\calX)$, and
    \item the $A^*(\calX)$-submodule generated by $\alpha_*[B\bmu_{b}]$ agrees with $\alpha_*A_*(B\bmu_{b})$.
\end{enumerate}
\end{prop}
\begin{proof}We first prove (1).
Consider $a:\mathbb{A}^1 \to \mathbb{A}^{n+1}$ the map sending $c\mapsto (c,0,...,0)$. We act with $\bG_m$ on $\mathbb{A}^1$ with weight $b$, and on $\mathbb{A}^{n+1}$ with weights $(b,a_1,...,a_n)$, so the map $a$ is equivariant. It induces a representable morphism of the corresponding quotient stacks, and a push-forward $A_*([\bA^1/\bG_m]) \to A_*([\bA^{n+1}/\bG_m])$.
Observe now that there is an open embedding $j: \calX\to [\bA^{n+1}/\bG_m]$,
and the following diagram is cartesian:
$$\xymatrix{B\bmu_b \ar[d]\ar[r] & \calX \ar[d]^{\text{open embedding}} \\ [\bA^1/\bG_m] \ar[r] & [\bA^{n+1}/\bG_m]}.$$
This induces the following diagram on Chow groups:
$$\xymatrix{A_*(B\bmu_b) \ar[r]^{\alpha_*}  & A_*(\calX)\\ A_*([\bA^1/\bG_m]) \ar[r]_{a_*} \ar[u]^{i^*}& A_*([\bA^{n+1}/\bG_m])\ar[u]_{j^*}}.$$
The vertical arrows are surjective, since they come from open embeddings, so to understand the image if $\alpha_*$ it suffices to understand the image of $a_*$ and $j^*$. 

 Consider then $U = \mathbb{A}^m\smallsetminus \{0\}$, with the action of $\mathbb{G}_m$ with weights 1. From Diagram \ref{comm diag}, it suffices to understand the image of $a^U:[\mathbb{A}^1\times U/\bG_m] \to [\mathbb{A}^{n+1}\times U/\bG_m]$ for every $n$. But
$[\mathbb{A}^1\times U/\bG_m] $ is the total space of $\calO_{\mathbb{P}^m}(b)$, $[\mathbb{A}^{n+1}\times U/\bG_m] $ is the total space of $\calO_{\mathbb{P}^m}(b,a_1,...,a_n)$, and the map is the inclusion in the first coordinate. Then
from
Lemma \ref{lemma_pull_back_class_vb}, $$a^U_*[[\mathbb{A}^1\times U/\bG_m]] = \phi^*(c_1(\calO_{\bP^m}(a_1))\cdot ... \cdot \phi^*(c_1(\calO_{\bP^m}(a_1))$$ where $\phi: \calO_{\mathbb{P}^m}(b,a_1,...,a_n) \to \bP^m$ is the projection. This holds for every $m$, so $$a_*[[\mathbb{A}^1/\bG_m]] = c_1(\calO_{[\mathbb{A}^{n+1}/\bG_m]}(a_1))\cdot ... \cdot c_1(\calO_{[\mathbb{A}^{n+1}/\bG_m]}(a_n)).$$Now (1) follows from the commutativity of the diagram above, since $i^*[\mathbb{A}^1/\bG_m] = [B\bmu_b]$ and $j^*c_1(\calO_{[\mathbb{A}^{n+1}\times U/\bG_m]}(a_i)) = c_1(\calO_{\calX}(a_i))$.

For proving (2) it suffices to check that the $A^*([\bA^{n+1}/\bG_m])$-submodule generated by $a_*([\bA^{1}/\bG_m])$ agrees with $a_*A_*([\bA^{1}/\bG_m])$. But as both $[\bA^{1}/\bG_m]$ and $[\bA^{n+1}/\bG_m]$ are vector bundles over $B\bG_m$, the pull-back $a^*:A^*([\bA^{n+1}/\bG_m])\to A^*([\bA^{1}/\bG_m])$ is an isomorphism. In particular it is surjective, so every $x\in A^*([\bA^{1}/\bG_m])$ can be written as $a^*(y)$ for $y\in A^*([\bA^{n+1}/\bG_m])$. Then point (2) follows from projection formula: $a_*a^*(y) = y\cdot a_*([\bA^{1}/\bG_m])$. 
\end{proof}

\begin{prop}\label{prop_chow_ring_of_U}
We have $A^*(\calU) = \mathbb{Z}[t]/24t^2$, where $t=c_1(\calO_{\calP}(1))_{|\calU}$. 
\end{prop}
\begin{proof}
Recall that the point $Z$ is of the form $[t^2,t^3,0]$.
Consider the inclusion of the cusp $C\to \mathbb{A}^3$ given by the points of the form $(t^2,t^3,0)$, composed with the seminormalization
$g:\mathbb{A}^1\to C\to \mathbb{A}^3$. Consider the action of $\bG_m$ on $\mathbb{A}^3$
with weights $2,3,4$ and on $\mathbb{A}^1$ with weight 1. Then $g$ is equivariant, and from Corollary \ref{cor_seminorm_is_chow_envelope_for_stacks} the following sequence is exact:
$$A_*([\mathbb{A}^1/\bG_m]) \xrightarrow{i} A_*([\mathbb{A}^3/\bG_m])\to A_*(\calP\smallsetminus Z) \to 0.$$ 
In particular, it suffices to understand the image of $i$. The argument is now very similar to the one of Proposition \ref{prop_inclusion_of_a_pt_with_some_coordinates_0_in_weighted_proj_stack}.

Consider then $U = \mathbb{A}^n\smallsetminus \{0\}$, with the action of $\mathbb{G}_n$ with weights 1. From Diagram \ref{comm diag}, it suffices to understand the map $j_U:[\mathbb{A}^1\times U/\bG_m] \to [\mathbb{A}^3\times U/\bG_m]$ for every $n$. But
$[\mathbb{A}^1\times U/\bG_m] $ is the total space of $\calO_{\mathbb{P}^n}(1)$, $[\mathbb{A}^3\times U/\bG_m] $ is the total space of $\calO_{\mathbb{P}^n}(2,3,4)$, and the map is of the type of
Corollary \ref{corollary_map_vector_bundles_in_chow}, with $a_1 = 2,$ $a_2 = 3$ and $V =\calO_{\mathbb{P}^n}(4) $.
Then from Corollary \ref{corollary_map_vector_bundles_in_chow}, $(j_U)_*[[\mathbb{A}^1\times U/\bG_m]] = 24p^*c_1(\calO(1))^2$ where $p:[\mathbb{A}^3\times U/\bG_m]\to \bP^n$ is the projection. Moreover, proceeding as in Proposition \ref{prop_inclusion_of_a_pt_with_some_coordinates_0_in_weighted_proj_stack}, one can check that $i(A_*([\mathbb{A}^1/\bG_m]))$ is generated by $i([[\mathbb{A}^1/\bG_m]])$.
Then the desired isomorphism
$A^*(\calU) \cong \mathbb{Z}[t]/24t^2$ holds since the argument above holds for every $n$.
\end{proof}

\begin{lemma}\label{lemma_self_intersection_exceptional}
Let $E$ be the class of the exceptional divisor of $\calM\to \calP$ in $A^*(\calM)$, and let $\iota:E\to \calM$ the inclusion of the exceptional divisor. Then $E^2 = \iota_*(c_1(\calO_E(-1)) $ and every element of $\iota_*A_*(E)$ is of the form $\sum_{i=1}^k a_i E^i$ for $a_i \in \mathbb{Z}$ and $k\in \mathbb{N}$.
\end{lemma}
\begin{proof}  
We have that $E^2 = \iota_*(c_1(\iota^*\calO_{\calM}(E)))$, so it suffices to check that $\iota^*\calO_{\calM}(E)) = \calO_E(-1)$. This can be carried out in a Zariski neighbourhood of $E$, so it follows from Observation \ref{obs_blowup_has_a_section}.

The second statement follows since from Lemma \ref{lemma_chow_ring_weighted_proj_space}, any element of $A_*(E)$ is of the form $a_0[E] + \sum_{i=1}^k a_i c_1(\calO_E(1))^i$. But $c_1(\calO(1)) = -c_1(\calO_E(-1)) =-c_1(\iota^*\calO_{\calM}(E)).$
\end{proof}
This is the last lemma we will need for determining the ring structure on $A^*(\overline{\calM}_{1,2})$.
\begin{lemma}\label{lemma_O(1)_restricts_to_0_on_E}
Let us denote by $\calL :=\calO_\calP(1)$ and with $y:= c_1(f^*(\calL))$. Then $E.y = 0$ in $A^*(\calM)$ and $\pi_*(y^2) = c_1(\calO_{\overline{\calM}_{1,1}}(1))$.
\end{lemma}
\begin{proof}
Observe that $E.y = 0$ since the restriction of $f^*(\calO_\calP(1))$ to $E$ is the trivial line bundle.

For the second statement, consider the section $s \in \operatorname{H}^0(\calP, \calO_\calP(2))$ with $V(s) = \{[0,a,b]\}\cong \calP(3,4)$. Consider the point $B\bmu_4 = [0,0,1] \in V(s)$, and let $\alpha:B\bmu_4\to V(s)$ be the inclusion.
From Proposition \ref{prop_inclusion_of_a_pt_with_some_coordinates_0_in_weighted_proj_stack}, $\alpha_*[B\bmu_4] = c_1(\calO_{\calP(3,4)}(3))$, and if $i:V(s)\to \calP$ is the inclusion of $V(s)$, we have $i^*c_1(\calO_\calP(1)) = c_1(\calO_{V(s)}(1))$. Therefore
$\alpha_*[B\bmu_4] = i^*c_1(\calO_\calP(3)).$

But recall that $Z$, the locus we blow-up, does not belong to $V(s)$, so we have a lifting of $i$ as follows:
$$\xymatrix{& & \calM\ar[d]^{f} \\
B\bmu_4 \ar[r]^{\alpha}& V(s)\ar[r]_i \ar[ur]^j &\calP}$$
Moreover $j(V(s)) = V(f^*(s))$. Then:

\begin{align*}
j_*\alpha_*[B\bmu_4] & = j_*i^* c_1(\calO_\calP(3))= j_*j^* f^*c_1(\calO_\calP(3))
= j_*j^* c_1(\calL^{\otimes 3}) = \\ &= c_1(\calL^{\otimes 3})\cdot j_*[V(s)]   = c_1(\calL^{\otimes 3})  \cdot c_1(\calL^{\otimes 2}) = 6 c_1(\calL)^2.
\end{align*}
Recall now that $\pi:\calM \to \calP(4,6)$ is a family of elliptic curves, so $\pi$ is representable, and therefore the morphism of groupoids $\calM(\Spec(k))\to \calP(4,6)(\Spec(k))$ is injective on automorphisms (\cite[Lemma 2.3.9]{AH}). Since the unique point in $B\bmu_4$ has $\bmu_4$ as automorphism group, and since the unique point of $\calP(4,6)$ with $\bmu_4$ as automorphism group is $[1,0]$, the composition
$\pi\circ \alpha \circ j$ agrees with the inclusion $B\bmu_4\hookrightarrow \calP(4,6)$ given by $[0,1]$. From Proposition \ref{prop_inclusion_of_a_pt_with_some_coordinates_0_in_weighted_proj_stack}, $(\pi\circ \alpha \circ j)_* = 6c_1(\calO_{\calP(4,6)}(1))$ so $\pi_*(6c_1(\calL)^2) = 6c_1(\calO_{\calP(4,6)}(1))$. From Lemma \ref{lemma_chow_ring_weighted_proj_space}, the group $A^1(\calP(4,6))$ is non-torsion, so $\pi_*(c_1(\calL)^2) = c_1(\calO_{\calP(4,6)}(1))$ as desired.
\end{proof}

\subsection{Chow rings of $\overline{\calM}_{1,2}$ and $\calM_{1,2}$.}\label{subsection_chow_ring_m12}In this subsection we compute $A^*(\overline{\calM}_{1,2})$ and $A^*(\calM_{1,2})$.
\subsubsection{Chow ring of $\overline{\calM}_{1,2}$.}
\begin{theorem}\label{thm_chow_ring_m12_bar}
We have $$A^*(\overline{\calM}_{1,2}) = \mathbb{Z}[x,y]/(24x^2+24y^2,xy)$$
where $y$ is the class of the exceptional divisor and $x$ the pull-back of $\calO_\calP(1)$ to $\calM$.
\end{theorem}
\begin{proof}
If $E$ is the exceptional divisor of $f$, then $E \cong \overline{\calM}_{1,1}$ and we have an exact sequence  
$$A_*(\overline{\calM}_{1,1}) \xrightarrow{\iota_*} A_*(\calM) \to A_*(\calU) \to 0 \text{ } \text{ } \text{ } \text{ } \text{ } (\ast).$$
Observe now that 
the map $\iota_*$ is split injective. In fact, the pushforward along the projection $\pi:\calM \to\overline{\calM}_{1,1}$ is a section of the pushforward along the inclusion of the exceptional divisor $\overline{\calM}_{1,1}\hookrightarrow \calM$.
In particular, we have an isomorphism $$\Phi:A_*(\calM) \xrightarrow{\pi_* \oplus j^*}  A_*(\overline{\calM}_{1,1}) \oplus A_*(\calU).$$
From Lemma \ref{lemma_chow_ring_weighted_proj_space}, $A^*(\overline{\calM}_{1,1}) = \mathbb{Z}[t]/24t^2$ and from Proposition \ref{prop_chow_ring_of_U}, $A^*(\calU) = \mathbb{Z}[x]/24x^2$.
 Therefore
for computing $A^*(\calM)$ we need to understand its intersection pairing. We will denote by $y$ the class of the exceptional divisor of $f$, with $\calL :=\calO_\calP(1)$ and with $x:= c_1(f^*(\calL))$.

First observe that $ \Phi(y^2) = (c_1\calO_{E}(-1), 0)$. Indeed, $j^*y = 0$ so also $j^*y^2 = 0$. On the other hand from Lemma \ref{lemma_self_intersection_exceptional} we have $y^2 = \iota_*c_1\calO_{E}(-1)$ where $\iota:E\to \calM$ is the inclusion of the exceptional divisor. But $\pi\circ \iota = \operatorname{Id}$ so $\pi_*y^2 =\pi_* \iota_*c_1\calO_{E}(-1) = c_1\calO_{E}(-1).$
From Lemma \ref{lemma_O(1)_restricts_to_0_on_E}, $xy = 0$ and $\pi_*x^2 = c_1\calO_{E}(1)$.

Observe now that $x$ and $y$ generate $A^*(\calM)$ as a graded $\mathbb{Z}$-algebra. Indeed it follows from the exact sequence $(\ast)$ that every element of $A_*(\calM)$ is of the form $f^*\xi + \iota_*\eta$. Moreover, every element in $f^*\xi$ is a power of $x$, whereas from Lemma \ref{lemma_self_intersection_exceptional} every element of the form $\iota_*\eta$ is a power of $y$. So we have a graded homomorphism $\Psi:\mathbb{Z}[x,y]\to A^*(\overline{\calM}_{1,2})$, sending $xy$ and $24x^2 + 24y^2$ to 0. To show that it induces the desired isomorphism, it suffices to show that
a homogeneous polynomial $p:=\sum_{i=0}^na_iy^ix^{n-1}$ such that $\Psi(p)=0$ belongs to the ideal $(xy,24x^2 + 24y^2)$. If $p$ has degree 2, the desired statement follows from the descriptions of $\iota(x^2),\iota(y^2), \pi(x^2), \pi(y^2)$ given above, and from $xy=0$. We focus now on the case $\operatorname{deg}(p)>2$.
Since $\Psi(xy)=0$, also $\Psi(a_0y^n + a_n x^n)=0$. Then also $j^*(\Psi(a_0y^n + a_n x^n)) = 0$, so $a_n$ is a multiple of 24: $a_n = 24\alpha_n$ for a certain $\alpha_n$. But when $n>2$ we have $A^{n-1}(\overline{\calM}_{1,1}) = \mathbb{Z}/24\mathbb{Z}$, therefore $\pi_*(a_nx^n) = 24\pi_*(\alpha_nx^n) = 0$. Then
$$0=\pi_* \Psi(a_0y^n + a_n x^n)=\pi_* \Psi(a_0y^n )+\pi_*\Psi( a_n x^n) = \pi_*\Psi( a_0 y^n).$$
But $\pi_*\Psi(a_0y^n) =a_0 \pi_*(y^{n}) = a_0\pi_* \iota_*(\iota^* y^{n-1}) =a_0\pi_* \iota_*(  c_1(\calO_{\overline{\calM}_{1,1}}(-1))^{n-1}) = a_0 c_1(\calO_{\overline{\calM}_{1,1}}(-1))^{n-1},$ (recall that $\pi\circ \iota = \operatorname{Id}$) so $\pi_*\Psi(a_0y^n)=0 $ implies $a_0 = 24\alpha_0$, as $A^{n-1}(\overline{\calM}_{1,1}) = \mathbb{Z}/24\mathbb{Z}$ and it is generated by $c_1(\calO_{\overline{\calM}_{1,1}}(-1))^{n-1}$.
Then $p = 24\alpha_0y^n + 24\alpha_nx^n + \sum_{i=1}^{n-1}a_ix^iy^{n-i}$: we have $p \in (xy,24x^2+24y^2)$.
\end{proof}

\subsubsection{Chow ring of $\calM_{1,2}$.}

We now shift our attention to $A^*(\calM_{1,2})$. Observe that $\calM_{1,2}\subseteq \calU$, and the complement of $\calM_{1,2}$ in $\overline{\calM}_{1,2}$ corresponds to the points parametrizing singular curves. Those are either non-irreducible curves (which are parametrized by the exceptional divisor of $f:\calM \to \calP(2,3,4)$), and the nodal cubic. From Lemma \ref{lemma_fibers_elliptic_curves_without_a_pt}, the complement of $\calM_{1,2}$ in $\calU$
is isomorphic to the fiber of $\overline{\calM}_{1,2} \to \overline{\calM}_{1,1}$ over the nodal cubic, without the marked point. In particular, it is isomorphic to the curve $[\Spec(k[x,y]/y^2-x^3 + 3x -2)/\bmu_2]$ (observe that $y^2-x^3 + 3x -2$ is a nodal cubic without a smooth point), with the action which sends $y\mapsto -y$. There are two fixed points, we will denote them with $p:=[1,0,-3]$ and $q:=[-2,0,-3]$.
\begin{lemma}\label{lemma_calP_minus_Bmu2_points_on_the_cusp}
Let $\calV:=\calP\smallsetminus \{p,q\}$. Then $A^*(\calV) = \mathbb{Z}[t]/12t^2$, where $t= c_1(\calO_\calP(1))$.
\end{lemma}
\begin{proof}Let $i:B\bmu_2 \to \calP$ be the inclusion of $[1,0,0]$, and $j^{(p)}$ (resp. $j^{(q)}$) the inclusion $B\bmu_2 \to \calP$ of $p$ (resp. $q$). 
Consider the maps $\sigma^{(p)}:\calP \to \calP$ and $\sigma^{(q)}:\calP \to \calP$, which send $[a,b,c]\mapsto [a,b,c+3a^2]$ and $[a,b,c]\mapsto [a,b,4c+3a^2]$. They are automorphisms such that $\sigma_i^*c_1(\calO_\calP(1)) = c_1(\calO_\calP(1))$, and $\sigma^{(p)}\circ j^{(p)} =\sigma^{(q)}\circ j^{(q)} =  i$. But then the following diagrams are cartesian, with the horizontal arrows being isomorphisms and the vertical ones proper:
$$\xymatrix{B\bmu_2 \ar[d]_-{j^{(p)}} \ar[r] & B\bmu_2 \ar[d]^i \\ \calP\ar[r]_-{\sigma^{(p)}} & \calP} \text{  }\text{  }\text{  }\text{  }\text{  }\text{  }\xymatrix{B\bmu_2 \ar[d]_-{j^{(q)}} \ar[r] & B\bmu_2 \ar[d]^i \\ \calP\ar[r]_-{\sigma^{(q)}} & \calP}$$
Then from Diagram \ref{comm diag}, $j^{(q)}_* = (\sigma^{(q)})^*\circ i_*$ and $j^{(p)}_* = (\sigma^{(p)})^*\circ i_*$. But $(\sigma^{(p)})^*$ and $(\sigma^{(q)})^*$ are the identity on $A^*(\calP)$, so $j^{(p)}_* = j^{(q)}_*$,
and from Proposition \ref{prop_inclusion_of_a_pt_with_some_coordinates_0_in_weighted_proj_stack} the submodule of $A^*(\calP)$ generated by $c_1(\calO_{\calP}(3))\cdot c_1(\calO_{\calP}(4))$ is $j^{(p)}_*(A_*(B\bmu_2))$. Now the desired statement follows from the following two exact sequences, and Lemma \ref{lemma_chow_ring_weighted_proj_space}
$$A_*(B\bmu_2) \to A_*(\calP) \to A_*(\calP\smallsetminus j^{(p)}(B\bmu_2)) \to 0 \text{ and }A_*(B\bmu_2) \to A_*(\calP\smallsetminus j^{(q)}(B\bmu_2)) \to A_*(\calV) \to 0$$
\end{proof}
\begin{lemma}\label{lemma_generators_chow_nodal_Stacky_curve}
Let $C:=V(y^2-x^3 + 3x -2)\subseteq \Spec(k[x,y])$, and $\calC:=[C/\bmu_2]$, where the action of $\bmu_2$ sends $y\mapsto -y$. Then $\calC$ has two
points $i^{(p)}:B\bmu_2 \to \calC$ and $i^{(q)}:B\bmu_2 \to \calC$ with non-trivial stabilizer, and $ \{[\calC],i^{(p)}_*A_*(B\bmu_2),i^{(q)}_*A_*(B\bmu_2)\}$ is a set of generators for $A_*(\calC)$.  
\end{lemma}
\begin{proof}
From the inclusion $i^{(p)}\sqcup i^{(q)}$ of $B\bmu^2 \sqcup B\bmu_2$, we have the exact sequence 
$$A_*(B\bmu_2)\oplus A_*(B\bmu_2) \xrightarrow{i^{(p)}\oplus i^{(q)}} A_*(\calC) \xrightarrow{j} A_*(\calC\smallsetminus \{p,q\})\to 0$$
so it suffices to show that the image of $[\calC]$ through $j$ generates $A_*(\calC\smallsetminus \{p,q\})$. But the two points with non-trivial stabilizers on $\calC$ correspond to the nodal point of $C$, and another smooth point. Consider the map $C\to \mathbb{A}^1$ that sends $(x,y)\mapsto x$. This is a $2:1$ finite map, that is a $\bmu_2$-\'etale cover over the $x$ such that $x^3-3x+2\neq 0$. Therefore $\calC\to \mathbb{A}^1$ is an isomorphism over the points with no stabilizer. Then
$\calC\smallsetminus \{p,q\}$ is an open subscheme $U$ of $\mathbb{A}^1$, so the desired statement follows since $A_*(\bA^1)$ is generated by $[\bA^1]$, and the (surjective) restriction $A_*(\bA^1)\to A_*(U)$ sends $[\bA^1]\mapsto [U]$.
\end{proof}
\begin{theorem}\label{theorem_chow_ring_m1,2} We have $A^*(\calM_{1,2})\cong\mathbb{Z}[t]/12t$ where $t$ is the restriction of $\calO_\calP(1)$ to $\calM_{1,2}$.
\end{theorem}
\begin{proof}
We have an exact sequence $A_*(\calC)\to A_*(\calU)\xrightarrow{i} A_*(\calM_{1,2})\to 0$
and from Lemma \ref{lemma_generators_chow_nodal_Stacky_curve} the kernel of $i$ is generated by the set $\{[\calC],j^{(p)}_*A_*(B\bmu_2),j^{(q)}_*A_*(B\bmu_2)\}.$ If we denote by $\overline{\calC}$ the closure of $\calC$ in $\calP$, an equation for $\overline{\calC}$ is $f := 4a_4^3 + 27(a_3^2-a_2^3 - a_2a_4)^2$ (this is the equation of the discriminant). Then $\calP\smallsetminus \overline{\calC} = [\mathbb{A}^3 \smallsetminus V(f) / \bG_m].$
We now compute $A^1([\mathbb{A}^3 \smallsetminus V(f) / \bG_m]) = \Pic([\mathbb{A}^3 \smallsetminus V(f) / \bG_m])$ as in  \cite[Proposition 6.3]{AI19}, using \cite[Proposition 2.10]{Bri13}. 

Notice that $\Pic([\bA^3/\bG_m]) =\Pic(B\bG_m)= \mathbb{Z}$. For $x \in \mathbb{A}^3\smallsetminus V(f)$, the character of $\bG_m$ defined by sending $\lambda \mapsto \frac{f(\lambda x)}{f(x)}$ is $\lambda\mapsto \lambda^{12}$.
 Then as in \cite[Proposition 6.3]{AI19}, we have $\Pic([\mathbb{A}^3 \smallsetminus V(f) / \bG_m]) = \mathbb{Z}/12\mathbb{Z}$, so if we denote by $t:= c_1(\calO_\calP(1)) \in A^*(\calU)$, we have $[\calC] = 12t$.
 
 From Lemma \ref{lemma_calP_minus_Bmu2_points_on_the_cusp}, $j^{(p)}_*A_*(B\bmu_2) = j^{(q)}_*A_*(B\bmu_2) = 12t^2$. Then the image of $A_*(\calC)\to A_*(\calU)$ is generated by $12t$, so $A^*(\calM_{1,2}) = \mathbb{Z}[t]/12t$.
\end{proof}
From Theorem \ref{theorem_chow_ring_m1,2} and Theorem \ref{thm_chow_ring_m12_bar} we deduce the following
\begin{cor}
The restriction map $A^*(\overline{\calM}_{1,2})\to A^*(\calM_{1,2})$ can be identified with the map $\mathbb{Z}[x,y]/(xy, 24x^2 + 24y^2)\to \mathbb{Z}[t]/(12t)$ that sends $y\mapsto 0$ and $x\mapsto t$.
\end{cor}

\bibliographystyle{amsalpha}
\bibliography{main}

\begin{bibdiv}
\begin{biblist}

\bib{AH}{incollection}{
      author={Abramovich, Dan},
      author={Hassett, Brendan},
       title={Stable varieties with a twist},
        date={2011},
   booktitle={Classification of algebraic varieties},
      series={EMS Ser. Congr. Rep.},
   publisher={Eur. Math. Soc., Z\"{u}rich},
       pages={1\ndash 38},
         url={https://doi.org/10.4171/007-1/1},
      review={\MR{2779465}},
}

\bib{AHR19}{article}{
      author={Alper, Jarod},
      author={Hall, Jack},
      author={Rydh, David},
       title={The étale local structure of algebraic stacks},
        date={2019},
     journal={arXiv preprint arXiv:1912.06162},
}

\bib{AI19}{article}{
      author={Asgarli, Shamil},
      author={Inchiostro, Giovanni},
       title={The {P}icard group of the moduli of smooth complete intersections
  of two quadrics},
        date={2019},
     journal={Transactions of the American Mathematical Society},
      volume={372},
      number={5},
       pages={3319\ndash 3346},
}

\bib{GMS_jarod}{article}{
      author={Alper, Jarod},
       title={Good moduli spaces for {A}rtin stacks},
        date={2013},
        ISSN={0373-0956},
     journal={Ann. Inst. Fourier (Grenoble)},
      volume={63},
      number={6},
       pages={2349\ndash 2402},
  url={http://aif.cedram.org.offcampus.lib.washington.edu/item?id=AIF_2013__63_6_2349_0},
      review={\MR{3237451}},
}

\bib{Brauer_group_m11}{article}{
      author={Antieau, Benjamin},
      author={Meier, Lennart},
       title={The {B}rauer group of the moduli stack of elliptic curves},
        date={2020},
        ISSN={1937-0652},
     journal={Algebra Number Theory},
      volume={14},
      number={9},
       pages={2295\ndash 2333},
  url={https://doi-org.offcampus.lib.washington.edu/10.2140/ant.2020.14.2295},
      review={\MR{4172709}},
}

\bib{ATW19}{article}{
      author={Abramovich, Dan},
      author={Temkin, Michael},
      author={W{\l}odarczyk, Jaros{\l}aw},
       title={Functorial embedded resolution via weighted blowings up},
        date={2019},
     journal={arXiv preprint arXiv:1906.07106},
}

\bib{Bri13}{article}{
      author={Brion, Michel},
       title={On linearization of line bundles},
        date={2015},
        ISSN={1340-5705},
     journal={J. Math. Sci. Univ. Tokyo},
      volume={22},
      number={1},
       pages={113\ndash 147},
      review={\MR{3329192}},
}

\bib{CDLI}{article}{
      author={Canning, Samir},
      author={Di~Lorenzo, Andrea},
      author={Inchiostro, Giovanni},
       title={The integral {C}how rings of moduli of {W}eierstrass fibrations},
        date={2022},
     journal={arXiv preprint arXiv:2204.05524},
}

\bib{DL}{article}{
      author={Di~Lorenzo, Andrea},
       title={The {C}how ring of the stack of hyperelliptic curves of odd
  genus},
        date={2021},
     journal={Int. Math. Res. Not. IMRN},
      number={4},
       pages={2642\ndash 2681},
}

\bib{DLFV}{article}{
      author={Di~Lorenzo, Andrea},
      author={Fulghesu, Damiano},
      author={Vistoli, Angelo},
       title={The integral {C}how ring of the stack of smooth non-hyperelliptic
  curves of genus three},
        date={2021},
     journal={Trans. Amer. Math. Soc.},
      volume={374},
      number={8},
       pages={5583\ndash 5622},
}

\bib{DLP}{article}{
      author={Di~Lorenzo, Andrea},
      author={Pirisi, Roberto},
       title={Brauer groups of moduli of hyperelliptic curves via cohomological
  invariants},
        date={2021},
     journal={Forum Math. Sigma},
      volume={9},
       pages={Paper No. e64, 37},
  url={https://doi-org.offcampus.lib.washington.edu/10.1017/fms.2021.55},
      review={\MR{4313563}},
}

\bib{di2021stable}{article}{
      author={Di~Lorenzo, Andrea},
      author={Pernice, Michele},
      author={Vistoli, Angelo},
       title={Stable cuspidal curves and the integral {C}how ring of
  {$\overline{\mathscr{M}}_{2,1}$}},
        date={2021},
     journal={arXiv preprint arXiv:2108.03680},
}

\bib{DLV}{misc}{
      author={Di~Lorenzo, Andrea},
      author={Vistoli, Angelo},
       title={Polarized twisted conics and moduli of stable curves of genus
  two},
        date={2021},
        note={available at \url{https://arxiv.org/abs/2103.13204}},
}

\bib{EdFu}{article}{
      author={Edidin, Dan},
      author={Fulghesu, Damiano},
       title={The integral {C}how ring of the stack of hyperelliptic curves of
  even genus},
        date={2009},
     journal={Math. Res. Lett.},
      volume={16},
      number={1},
       pages={27\ndash 40},
}

\bib{EG98}{article}{
      author={Edidin, Dan},
      author={Graham, William},
       title={Equivariant intersection theory},
        date={1998},
        ISSN={0020-9910},
     journal={Invent. Math.},
      volume={131},
      number={3},
       pages={595\ndash 634},
  url={https://doi-org.offcampus.lib.washington.edu/10.1007/s002220050214},
      review={\MR{1614555}},
}

\bib{FringuelliPirisi}{article}{
      author={Fringuelli, Roberto},
      author={Pirisi, Roberto},
       title={The {B}rauer group of the universal moduli space of vector
  bundles over smooth curves},
        date={2021},
        ISSN={1073-7928},
     journal={Int. Math. Res. Not. IMRN},
      number={18},
       pages={13609\ndash 13644},
  url={https://doi-org.offcampus.lib.washington.edu/10.1093/imrn/rnz300},
      review={\MR{4320792}},
}

\bib{Ful13}{book}{
      author={Fulton, William},
       title={Intersection theory},
   publisher={Springer Science \& Business Media},
        date={2013},
      volume={2},
}

\bib{Ger17}{article}{
      author={Geraschenko, Anton},
      author={Satriano, Matthew},
       title={A “bottom up” characterization of smooth deligne--mumford
  stacks},
        date={2017},
     journal={International Mathematics Research Notices},
      volume={2017},
      number={21},
       pages={6469\ndash 6483},
}

\bib{Lar}{article}{
      author={Larson, Eric},
       title={The integral {C}how ring of {$\overline{ M}_2$}},
        date={2021},
     journal={Algebr. Geom.},
      volume={8},
      number={3},
       pages={286\ndash 318},
}

\bib{Lieb08}{article}{
      author={Lieblich, Max},
       title={Twisted sheaves and the period-index problem},
        date={2008},
     journal={Compositio Mathematica},
      volume={144},
      number={1},
       pages={1\ndash 31},
}

\bib{period_index_max}{article}{
      author={Lieblich, Max},
       title={Period and index in the {B}rauer group of an arithmetic surface},
        date={2011},
        ISSN={0075-4102},
     journal={J. Reine Angew. Math.},
      volume={659},
       pages={1\ndash 41},
  url={https://doi-org.offcampus.lib.washington.edu/10.1515/CRELLE.2011.059},
        note={With an appendix by Daniel Krashen},
      review={\MR{2837009}},
}

\bib{Mas14}{article}{
      author={Massarenti, Alex},
       title={The automorphism group of {$\overline{M}_{g,n}$}},
        date={2014},
     journal={Journal of the London Mathematical Society},
      volume={89},
      number={1},
       pages={131\ndash 150},
}

\bib{MV06}{article}{
      author={Molina~Rojas, Luis~Alberto},
      author={Vistoli, Angelo},
       title={On the chow rings of classifying spaces for classical groups},
        date={2006},
     journal={Rendiconti del Seminario Matematico della Universit{\`a} di
  Padova},
      volume={116},
       pages={271\ndash 298},
}

\bib{Ols05}{article}{
      author={Olsson, Martin},
       title={On proper coverings of {A}rtin stacks},
        date={2005},
     journal={Advances in Mathematics},
      volume={198},
      number={1},
       pages={93\ndash 106},
}

\bib{Ols16}{book}{
      author={Olsson, Martin},
       title={Algebraic spaces and stacks},
      series={American Mathematical Society Colloquium Publications},
   publisher={American Mathematical Society, Providence, RI},
        date={2016},
      volume={62},
        ISBN={978-1-4704-2798-6},
}

\bib{Minseon}{article}{
      author={Shin, Minseon},
       title={The cohomological {B}rauer group of weighted projective spaces
  and stacks},
        date={2020},
     journal={arXiv preprint arXiv:2009.14309},
}

\bib{Silverman}{book}{
      author={Silverman, Joseph~H},
       title={The arithmetic of elliptic curves},
   publisher={Springer Science \& Business Media},
        date={2009},
      volume={106},
}

\bib{stacks-project}{misc}{
      author={{Stacks project authors}, The},
       title={The stacks project},
         how={\url{https://stacks.math.columbia.edu}},
        date={2021},
}

\bib{Vis}{article}{
      author={Vistoli, Angelo},
       title={The {C}how ring of {$\scr M_2$}. {A}ppendix to ``{E}quivariant
  intersection theory'' [{I}nvent. {M}ath. {\bf 131} (1998), no. 3, 595--634;
  {MR}1614555 (99j:14003a)] by {D}. {E}didin and {W}. {G}raham},
        date={1998},
        ISSN={0020-9910},
     journal={Invent. Math.},
      volume={131},
      number={3},
       pages={635\ndash 644},
  url={https://doi-org.offcampus.lib.washington.edu/10.1007/s002220050215},
      review={\MR{1614559}},
}

\bib{Wei91}{article}{
      author={Weibel, Charles~A},
       title={Pic is a contracted functor},
        date={1991},
     journal={Inventiones mathematicae},
      volume={103},
      number={1},
       pages={351\ndash 377},
}

\end{biblist}
\end{bibdiv}
\end{document}